\documentclass[10pt, a4paper]{article}
\usepackage{graphicx,latexsym, amsmath,amsfonts}
\usepackage{amsthm}
\usepackage{enumerate}
\usepackage{bbm}
\usepackage{subfig}

\usepackage{epsfig}
\usepackage{color}

\begin{document}

	
	\renewcommand{\d}{d}
	
	\newcommand{\E}{\mathbb{E}}
	\newcommand{\PP}{\mathbb{P}}
	\newcommand{\DD}{\mathbb{D}}
	\newcommand{\R}{\mathbb{R}}
	\newcommand{\cD}{\mathcal{D}}
	\newcommand{\cF}{\mathcal{F}}
	\newcommand{\cK}{\mathcal{K}}
	\newcommand{\N}{\mathbb{N}}
	\newcommand{\fracs}[2]{{ \textstyle \frac{#1}{#2} }}
	\newcommand{\sign}{\operatorname{sign}}
	
	\newtheorem{theorem}{Theorem}[section]
	\newtheorem{lemma}[theorem]{Lemma}
	\newtheorem{coro}[theorem]{Corollary}
	\newtheorem{defn}[theorem]{Definition}
	\newtheorem{assp}[theorem]{Assumption}
	\newtheorem{expl}[theorem]{Example}
	\newtheorem{prop}[theorem]{Proposition}
	\newtheorem{proposition}[theorem]{Proposition}
	\newtheorem{corollary}[theorem]{Corollary}
	\newtheorem{rmk}[theorem]{Remark}
	\newtheorem{notation}[theorem]{Notation}

	\def\a{\alpha} \def\g{\gamma}
	\def\e{\varepsilon} \def\z{\zeta} \def\y{\eta} \def\o{\theta}
	\def\vo{\vartheta} \def\k{\kappa} \def\l{\lambda} \def\m{\mu} \def\n{\nu}
	\def\x{\xi}  \def\r{\rho} \def\s{\sigma}
	\def\p{\phi} \def\f{\varphi}   \def\w{\omega}
	\def\q{\surd} \def\i{\bot} \def\h{\forall} \def\j{\emptyset}
	
	\def\be{\beta} \def\de{\delta} \def\up{\upsilon} \def\eq{\equiv}
	\def\ve{\vee} \def\we{\wedge}
	
	\def\t{\tau}
	
	\def\F{{\cal F}}
	\def\T{\tau} \def\G{\Gamma}  \def\D{\Delta} \def\O{\Theta} \def\L{\Lambda}
	\def\X{\Xi} \def\S{\Sigma} \def\W{\Omega}
	\def\M{\partial} \def\N{\nabla} \def\Ex{\exists} \def\K{\times}
	\def\V{\bigvee} \def\U{\bigwedge}
	
	\def\1{\oslash} \def\2{\oplus} \def\3{\otimes} \def\4{\ominus}
	\def\5{\circ} \def\6{\odot} \def\7{\backslash} \def\8{\infty}
	\def\9{\bigcap} \def\0{\bigcup} \def\+{\pm} \def\-{\mp}
	\def\<{\langle} \def\>{\rangle}
	
	\def\lev{\left\vert} \def\rev{\right\vert}
	\def\1{\mathbf{1}}

	\newcommand\wD{\widehat{\D}}
	\newcommand\EE{\mathbb{E}}
	
	\newcommand{\ls}[1]{\textcolor{red}{\tt Lukas: #1 }}

	\title{ \bf The  weak convergence order of two  Euler-type discretization schemes for the log-Heston model}

	\author{Annalena Mickel \footnote{Mathematical Institute and   DFG Research Training Group 1953, 
			University of Mannheim,  B6, 26, D-68131 Mannheim, Germany,  \texttt{amickel@mail.uni-mannheim.de}}
		\and   Andreas Neuenkirch  \footnote{Mathematical Institute, 
			University of Mannheim,  B6, 26, D-68131 Mannheim, Germany, \texttt{aneuenki@mail.uni-mannheim.de}} 
	}

	\date{\today}

	\maketitle
	
	\begin{abstract} 
		We study the weak convergence order of two Euler-type discretizations of the log-Heston Model where we use symmetrization and absorption, respectively, to prevent the discretization of the underlying CIR process from becoming negative. 
		If the  Feller index $\nu$  of the CIR process satisfies $\nu>1$, we establish weak convergence order  one, while for $\nu \leq 1$, we obtain weak convergence order $\nu-\epsilon$ for $\epsilon>0$ arbitrarily small.
		We illustrate our theoretical findings by several numerical examples.

		\medskip
		\noindent \textsf{{\bf Key words: } \em Heston model,  CIR process, discretization schemes for SDEs, Kolmogorov PDE, Euler scheme, non-standard assumptions \\
		}
		
	\end{abstract}

	\section{Introduction and Main Result}
	The Heston model \cite{heston} is given by the stochastic differential equation (SDE)
	\begin{eqnarray}  \begin{split}
			\d S_t &= \mu S_t \d t + \sqrt{V_t} S_t  \Big(\rho \d W_t + \sqrt{1-\rho^2} \d B_t\Big), \qquad  & S_0=s_0, \\
			\d V_t &= \kappa(\theta - V_t) \d t + \sigma \sqrt{V_t} \d W_t,  \qquad  \qquad  \qquad \qquad  \,\,  & V_0=v_0,\label{hes-eq}
		\end{split}
	\end{eqnarray}
	with $s_0, v_0, \kappa, \theta, \sigma >0$, $\mu \in \mathbb{R}$,  $\rho \in [-1,1]$, $T>0$ and independent Brownian motions  $W=(W_t)_{t \in [0,T]}$,  $B=(B_t)_{t \in [0,T]}$ which are defined on a filtered probability space $(\Omega, \cF, (\cF_t)_{t\in[0,T]}, \mathbb{P})$ where the filtration satisfies the usual conditions. Throughout this article the initial values $s_0$, $v_0$ are assumed to be deterministic. 
	Here, $S$ models the price of an asset and $V$ its volatility which is given by the so called  Cox--Ingersoll--Ross process (CIR).
	
	The Heston model extends the celebrated Black-Scholes model by using a stochastic volatility instead of a constant one. Its numerical analysis  is not only of theoretical relevance. With the recent rise of volatility trading in financial markets, stochastic volatility models as the Heston model are becoming more and more important and numerical calculations are essential for the practical use of these models.

	To price options with maturity at time $T$  and discounted payoff function $g:[0,\infty) \rightarrow \mathbb{R}$, one is interested in the value of
	\begin{align*}
		\mathbb{E}\left[g(S_T)\right],
	\end{align*}
	which can be approximated through a (multi-level) Monte Carlo simulation.
	Usually, the log-Heston model  instead of the Heston model is considered in numerical practice. With $X_t=\log(S_t)$ this yields the SDE
	\begin{align} \begin{split}
			\d X_t &= \Big(\mu-\frac12 V_t\Big ) \d t + \sqrt{V_t} \Big(\rho \d W_t + \sqrt{1-\rho^2} \d B_t\Big), \\
			\d V_t &= \kappa (\theta-V_t) \d t + \sigma \sqrt{V_t} \d W_t, \label{hes-log} \end{split}
	\end{align}
	and the exponential is then incorporated in the payoff, i.e.\ $g$ is replaced by
	$h: \mathbb{R} \rightarrow  \mathbb{R}$ with $h(x)=g(\exp(x))$.
	
	Note that the state space of the log-Heston SDE \eqref{hes-log}  is $\mathbb{R} \times [0, \infty)$ and the square-root coefficients are not globally Lipschitz continuous.  Additionally, even if $g$ satisfies a linear growth condition, $g(S_T)$ might not belong to $L^p$ if $p>1$. See e.g. \cite{And-M} for this so-called  moment explosion. Thus, the (log-)Heston SDE does not satisfy the standard assumptions for the numerical analysis of SDEs.
	
	Due to its popularity and its difficulty various methods for the simulation of the (log-)Heston model have been proposed and studied:
	\begin{itemize}
		\item[(i)] exact simulations schemes which are based on the transition density of the CIR process and the characteristic function of the log-asset price process,  see e.g. \cite{BK,Smith,GlaKim, Wiese, Grezlak, Wiese2},
		\item[(ii)] semi-exact discretization schemes which simulate the CIR process exactly, but discretize the log-asset price, see e.g. \cite{And-H,zheng1,MN},
		\item[(iii)] full discretization schemes which discretize both components of the SDE, see   e.g.  \cite{Higham,KJ,And-H, NV, NN, lord-comparison, AltNeu}, 
		\item[(iv)] and miscellaneous methods as e.g. tree methods in \cite{tree1,Kirkby2} or the moment matching methods in \cite{And-H}.
	\end{itemize}

	Euler-type discretization schemes are frequently used  for the log-Heston model and seem to be efficient in many cases, but a weak error analysis, i.e. an analysis of the bias of these discretization schemes, is not available yet. Since the discretization of the CIR process might become negative, these schemes require fixes as absorption and symmetrization or a modification of the square-root coefficient for negative values which make these schemes applicable but  their analysis even more challenging. For a survey of such fixes and modifications see e.g. \cite{lord-comparison}. Moreover, for the importance of the weak error in (multi-level) Monte Carlo simulations see \cite{duffie_glynn} and \cite{giles}.

	Here, we will use two of these schemes for the discretization of the CIR component. The Symmetrized Euler (SE) scheme given  by
	\begin{equation}
		\begin{aligned}
			v^{sym}_{k+1}&=\biggl|v^{sym}_k+\kappa\left(\theta-v^{sym}_k\right)(t_{k+1}-t_k)+\sigma\sqrt{v^{sym}_k}\left(W_{t_{k+1}}-W_{t_k}\right)\biggr|,
		\end{aligned}
	\end{equation}
	where $0=t_0<t_1<...<t_k<...<t_N=T$ and $v^{sym}_{0}=v_0$, has been {{proposed and studied by  Bossy and Diop in \cite{BO}. They state the weak convergence order one (for  sufficiently regular test functions) under the assumption $ \frac{\kappa \theta}{\sigma^2} >1$ for this scheme.}}
	Instead of reflecting negative values, the Euler scheme with absorption or Absorbed  Euler (AE) scheme sets all occurring negative values of the  CIR discretization to zero, i.e.
	\begin{equation}
		v^{abs}_{k+1}=\left(v^{abs}_k+\kappa\left(\theta-v^{abs}_k\right)(t_{k+1}-t_k)+\sigma\sqrt{v^{abs}_k}\left(W_{t_{k+1}}-W_{t_k}\right)\right)^{+}
	\end{equation}
	and $v^{abs}_{0}=v_0$. {{
			Its exact origin is unknown and also no weak convergence rates are known for it.}}
	Using SE or AE, we  then discretize the log-price process with the standard Euler method:
	\begin{equation}\label{disc-heston}
		x_{k+1}=x_k+ \Big( \mu - \frac{1}{2} v_k \Big)(t_{k+1}-t_k)+\sqrt{v_k}\Big(\rho\left(W_{t_{k+1}}-W_{t_k}\right)+\sqrt{1-\rho^2}\left(B_{t_{k+1}}-B_{t_k}\right)\Big).
	\end{equation}

	We will work with an equidistant discretization
	$$ t_k = k \Delta t, \quad k=0, \ldots, N,$$
	with $\Delta t= T/ N$ and $N \in \mathbb{N}$.
	Moreover, we will use the standard notations for the spaces of differentiable functions. In particular, the subscript $pol$ denotes polynomial growth. See also Subsection \ref{aux_pde}.
	Finally, we set
	$$ \nu = \frac{2\kappa\theta}{\sigma^2},$$
	i.e. $\nu$ is the Feller index of the CIR process $V=(V_t)_{t \in [0,T]}$.
	
	Our analysis leads us to the following theorem:
	\begin{theorem}\label{maintheorem} (i) Let $f\in C^{6}_{pol}(\R \times [0,\infty))$ and $\nu>1$.
		Then both schemes satisfy
		\begin{align*}
			\limsup\limits_{N \rightarrow \infty }  \, \frac{ \left|\mathbb{E}\left[f(x_N,v_N)\right]-\mathbb{E}\left[f(X_T,V_T)\right]\right|}{ \left| T/N \right| } < \infty.
		\end{align*} (ii)   Let $f\in C^{6}_{pol}(\R \times [0,\infty))$ and $\nu \leq 1$. 	Then both schemes satisfy
		\begin{align*}
			\limsup\limits_{N \rightarrow \infty }  \, \frac{ \left|\mathbb{E}\left[f(x_N,v_N)\right]-\mathbb{E}\left[f(X_T,V_T)\right]\right|}{ \left| T/N \right|^{\alpha}}=0
		\end{align*}
		for all $\alpha\in (0, \nu)$.
	\end{theorem}
	
	Thus, for $\nu > 1$ we have weak convergence order one  and for $\nu \leq 1$ we have weak convergence order  $\nu-\epsilon$ for arbitrarily small $\epsilon >0$. In view of these convergence rates, exact and semi-exact simulation schemes are relevant alternatives for small Feller indices $\nu$.
	
	\smallskip

	\begin{rmk}
		Our analysis {{builds on the regularity of the Heston PDE established in \cite{BCT}. Moreover, we use the local time-approach for the analysis of the discretization of the CIR process from \cite{BO} and we modify and improve the analysis of the probability that the discretization of the CIR process becomes negative from \cite{BO,CS2}. In particular, we are able to remove the restriction $\nu >2$ for the Feller index in comparison to  \cite{BO,CS2}.}}
		
		As already mentioned, \cite{BO} studies the weak error of the symmetrized Euler scheme for the CIR process. In \cite{CS2} the strong error of a fully truncated Euler scheme for the CIR process is analyzed, while \cite{BCT} studies a hybrid-tree scheme for option pricing in stochastic volatility models with jumps.

	\end{rmk}
	
	\begin{rmk} The (positivity preserving) weak approximation of the CIR process has been also studied by Alfonsi in \cite{AA1,AA2}. In particular, weak first and second order schemes have been derived in these references. {{The weak approximation of the log-Heston model by a combined Euler scheme for $X_t$ and  a drift implicit Milstein scheme for $V_t$ has been studied in \cite{AltDiss,AltNeu}. In these works weak convergence order one is obtained for smooth test functions $f$ if $\nu >2$ and for bounded and measurable  test functions $f$ if  $\nu > \frac{9}{2} $, respectively.}}

		Another approach for the weak approximation of the CIR process and the Heston model, which combines the cubature on Wiener space approach and the classical Runge-Kutta method for ordinary differential equations, has been introduced in \cite{NN}; see also the related work \cite{NV} and Remark \ref{rem_bally}.
	\end{rmk}

	\begin{rmk} A decay of the weak convergence rate for $\nu <1$ has also been numerically
		observed  in \cite{BK} for an Euler scheme for the Heston Model and  for some of the schemes in \cite{AA1} for the CIR process.
		Interestingly, the convergence order $\nu$  also appears for the CIR process  in a different context, namely for the $L^1$-approximation at the terminal time point. If $\nu<1$, then $\nu$ is the best possible convergence order that  arbitrary approximations based on an equidistant discretization of the driving Brownian motion can achieve, see \cite{JentzenHefter}. Together with \cite{HEFTER} and \cite{MiNe} the latter reference also  provides  a comprehensive survey on the strong approximation of the CIR process.
	\end{rmk}
	
			\begin{rmk} The Heston model was introduced in the 1990s and is now a classical stochastic volatility model; our analysis addresses a problem that has been open for a long time. Meanwhile, the Heston model has been extended in several ways, e.g., by adding jumps \cite{Bates,BCT} or by considering rough volatility processes instead of the CIR process, which leads to the rough Bergomi model \cite{bergomi} or the rough Heston model \cite{rHeston}.
	\end{rmk}
	\bigskip

	\section{Numerical results}
	In this section, we will test numerically whether the weak convergence rates of Theorem \ref{maintheorem} are attained even under milder assumptions on the test function $f$. 
	We will consider a call, a put and a digital option. These payoffs are at most Lipschitz continuous which is typical in financial applications. This lack of smoothness is in contrast to the usual assumptions on $f$ for a weak error analysis. See also Remark \ref{rem_bally}.

	We use the following model parameters. 
	\begin{enumerate}
		\item[Model 1:] $S_0=100, V_0=0.04, K=100, \kappa = 5, \theta=0.04, \sigma=0.61, \rho=-0.7, T=1, r=0.0319$
		\item[Model 2:] $S_0=100, V_0=0.0457, K=100, \kappa = 5.07, \theta=0.0457, \sigma=0.48, \rho=-0.767, T=2, r=0.0$
		\item[Model 3:] $S_0=100, V_0=0.010201, K=100, \kappa = 6.21, \theta=0.019, \sigma=0.61, \rho=-0.7, T=1, r=0.0319$
		\item[Model 4:] $S_0=100, V_0=0.09, K=100, \kappa = 2, \theta=0.09, \sigma=1, \rho=-0.3, T=5, r=0.05$
	\end{enumerate}
	With these examples, we try to cover a wide range of Feller indices. We have $\nu\approx 1.075$ in Model 1, $\nu\approx2.01$ in Model 2, $\nu\approx0.63$ in Model 3 and $\nu\approx0.36$ in Model 4. For each model, we use the following discounted payoff functions:
	\begin{enumerate}
		\item European call: $g_1(S_T)= e^{-rT}\max\{S_T-K,0\}$
		\item European put: $g_2(S_T)= e^{-rT}\max\{K-S_T,0\}$
		\item Digital option: $g_3(S_T)= e^{-rT}\mathbbm{1}_{[0,K]}(S_T)$
	\end{enumerate}
	{Note that none of these payoffs satisfies the assumption of our Theorem. Thus, numerical convergence rates which coincide with the rates of our Theorem indicate that the latter might be valid under milder assumptions.}

	In order to measure the weak error rate, we simulated $M=2\cdot10^7$ independent copies  $g_i(s_N^{(j)})$, $j=1, \ldots, M$, of $g_i\left(s_N\right)$ with $s_N=\exp(x_N)$ to estimate $$ \mathbb{E}(g_i(s_N))$$ by $$p_{M,N}=\frac{1}{M} \sum_{j=1}^M g_i(s_N^{(j)})$$  for each combination of model parameters, payoff and number of steps $N\in\{2^3,...,2^7\}$ or $N\in\{2^3,...,2^8\}$ where $\Delta t=\frac{T}{N}$. To obtain a stable estimate of the convergence rates, we started with a $\Delta t$ which is roughly around $\frac{1}{\kappa}$ (which is required also for some auxiliary results of the proof of our main result). The Monte Carlo mean of these
	samples is then compared to a reference solution $p_{\sf ref}$, i.e.,
			$$ e(N)=| p_{\sf ref}-p_{M,N}|, $$
			and we measure the weak error order by the slope of a least-squares fit of the data $(N,\log_2(e(N))$. The reference solutions can be computed with sufficiently high accuracy from semi-explicit formulae via Fourier methods. In particular, the put price can be calculated from the call pricing formula given in \cite{heston} via the put-call-parity. The price of the digital option can be computed from the probability $P_2$ given in \cite{heston}; it equals $e^{-rT}\left(1-P_2\right)$. 
	
	In Table \ref{table:1} we can see the measured convergence rates for Model 1. In Figure \ref{fig1a}, the error plot for the call option is displayed\footnote{The plots for the different options behave similarly, i.e. regularly and in accordance with Theorem  \ref{maintheorem}, so we show only one plot for each option.}.
	\begin{table}[htb!]
		\centering {
				\begin{tabular}{c c c c} 
					\hline
					Method & Call & Put & Digital\\ [0.5ex] 
					\hline
					SE & 1.02 & 1.00 & 0.94\\ 
					AE & 1.02 & 0.94 & 0.91\\
					\hline
				\end{tabular}
				\caption{Measured convergence rates Model 1}
				\label{table:1} }
	\end{table}
	Because of our results in Theorem \ref{maintheorem}, we would expect SE  and AE to have a weak convergence rate of $1$ and this is indeed the case in this example.
	\begin{figure}[htb!]
		\centering
		\includegraphics[scale=.15]{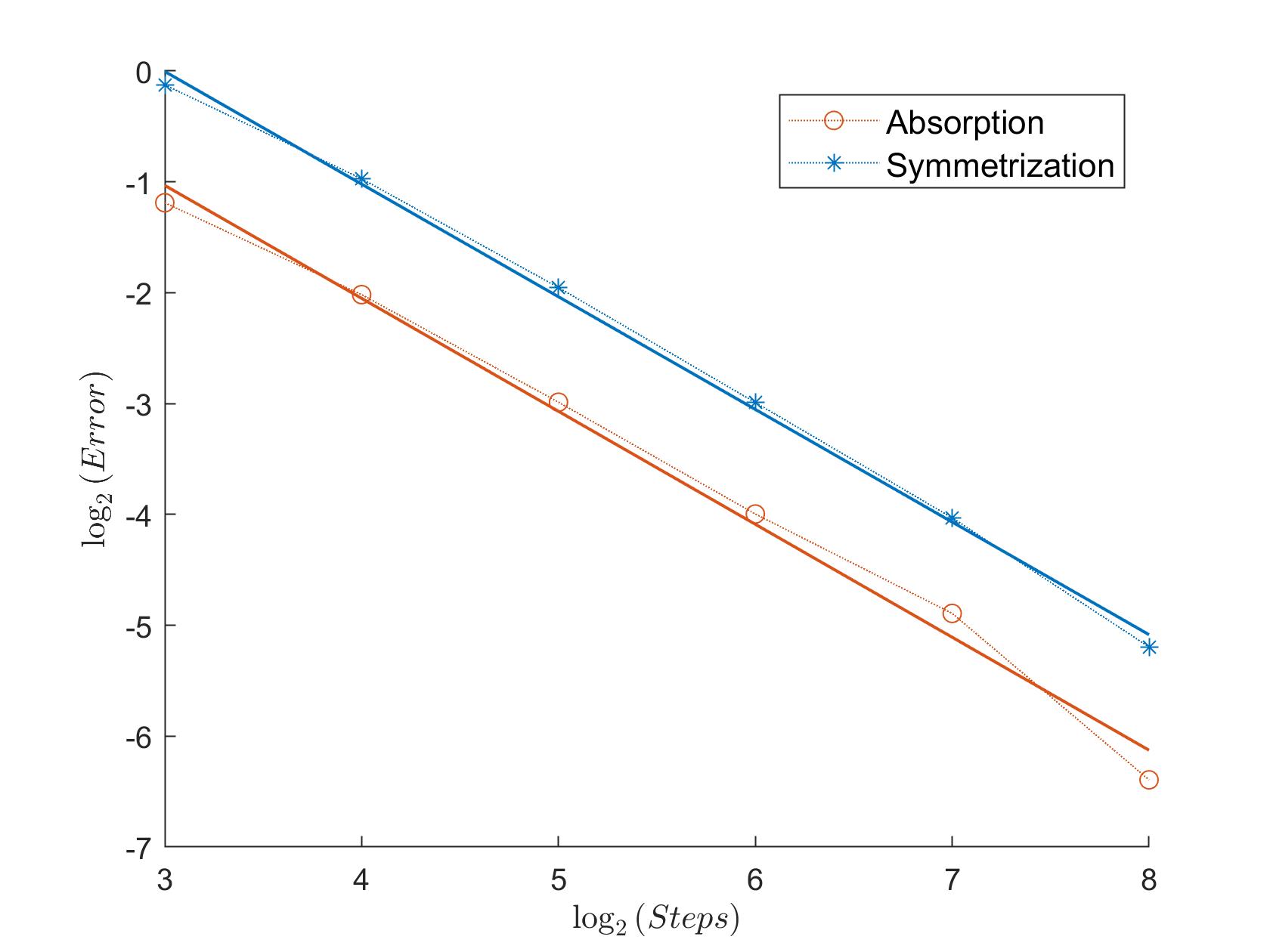}
		\caption{Call Model 1}
		\label{fig1a}
	\end{figure}
	Looking at the plot, we can see that the convergence behavior is very regular. 
	
	For the next model, we would again expect a convergence rate around $1$.  The results from Table \ref{table:2} indicate that for particular payoffs even a higher numerical rate is obtained if the Feller index is larger than $1$. The rates of SE and AE are around $1.5$. The error plot of the put option in Figure \ref{fig2a} shows again a regular convergence behavior.
	\begin{table}[htb!]
		\centering {
				\begin{tabular}{c c c c} 
					\hline
					Method & Call & Put & Digital\\ [0.5ex] 
					\hline
					SE & 1.51 & 1.50 & 1.34\\ 
					AE & 1.37 & 1.50 & 1.42\\
					\hline
				\end{tabular}
				\caption{Measured convergence rates Model 2}
				\label{table:2} }
	\end{table} 
	\begin{figure}[htb!]
		\centering
		\includegraphics[scale=.15]{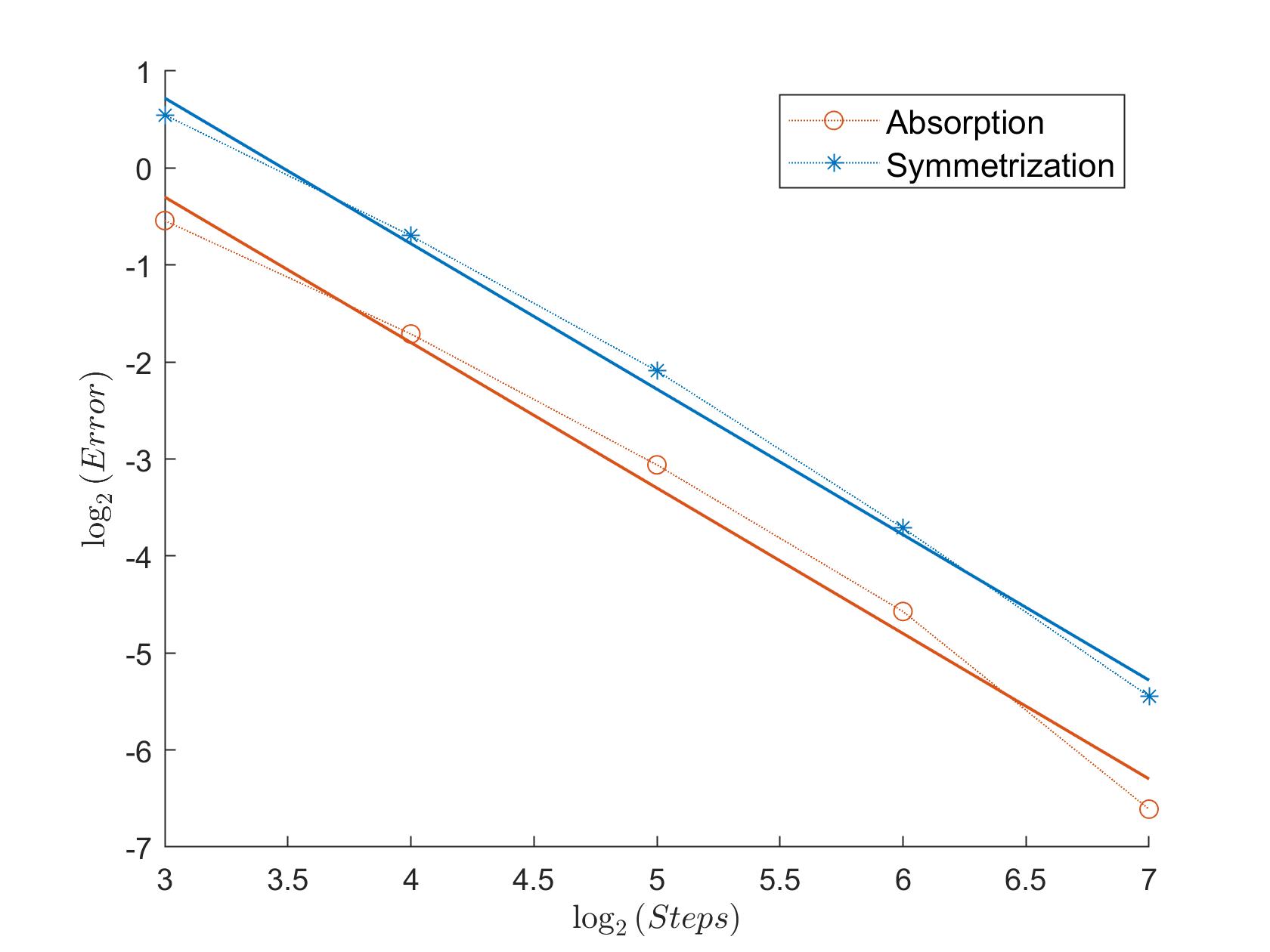}
		\caption{Put Model 2}
		\label{fig2a}
	\end{figure}
	Table \ref{table:3} shows the estimated convergence rates for Model 3. This model has a Feller index around $0.63$. The simulation results indicate that this is also the convergence rate for SE and AE. 
	\begin{table}[htb!]
		\centering {
				\begin{tabular}{c c c c} 
					\hline
					Method & Call & Put & Digital\\ [0.5ex] 
					\hline
					SE & 0.60 & 0.60 & 0.55 \\ 
					AE & 0.57 & 0.57 & 0.55\\
					\hline
				\end{tabular}
				\caption{Measured convergence rates Model 3}
				\label{table:3} }
	\end{table} 
	The error plot of the digital option in Figure \ref{fig3a} again shows no irregularities.
	\begin{figure}[htb!]
		\centering
		\includegraphics[scale=.15]{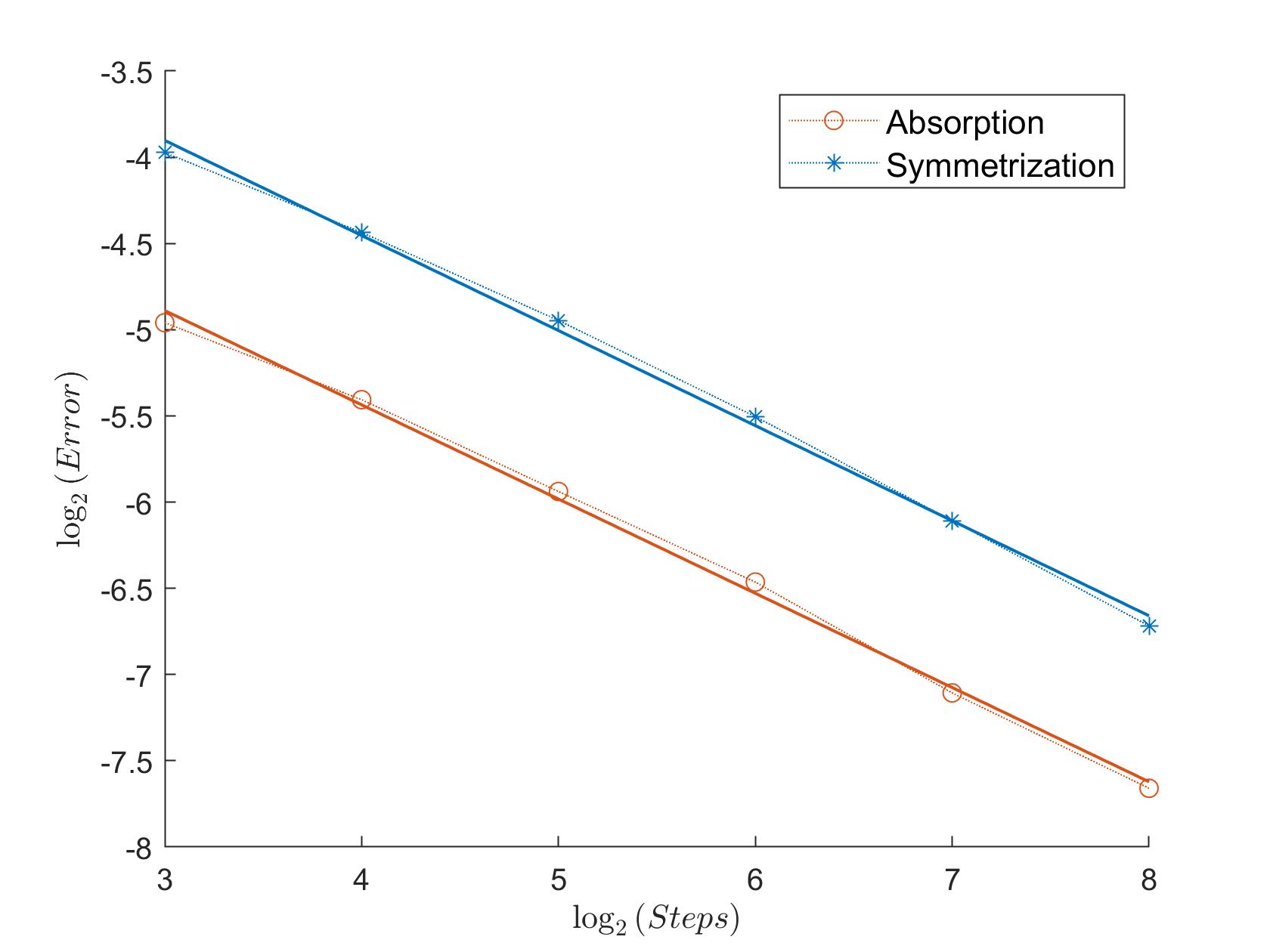}
		\caption{Digital Model 3}
		\label{fig3a}
	\end{figure}
	
	Model 4 has the lowest Feller index which is around $0.36$. Again, Table \ref{table:4} confirms this number as the numerical convergence rate for SE and AE and the plot in Figure \ref{fig4a}  shows a regular behavior.
	\begin{table}[htb!]
		\centering {	
				\begin{tabular}{c c c c} 
					\hline
					Method & Call & Put & Digital\\ [0.5ex] 
					\hline
					SE & 0.47 & 0.47 & 0.40\\ 
					AE & 0.39 & 0.39 & 0.35\\
					\hline
				\end{tabular}
				\caption{Measured convergence rates Model 4}
				\label{table:4} }
	\end{table} 
	
	\begin{figure}[htb!]
		\centering
		\includegraphics[scale=.15]{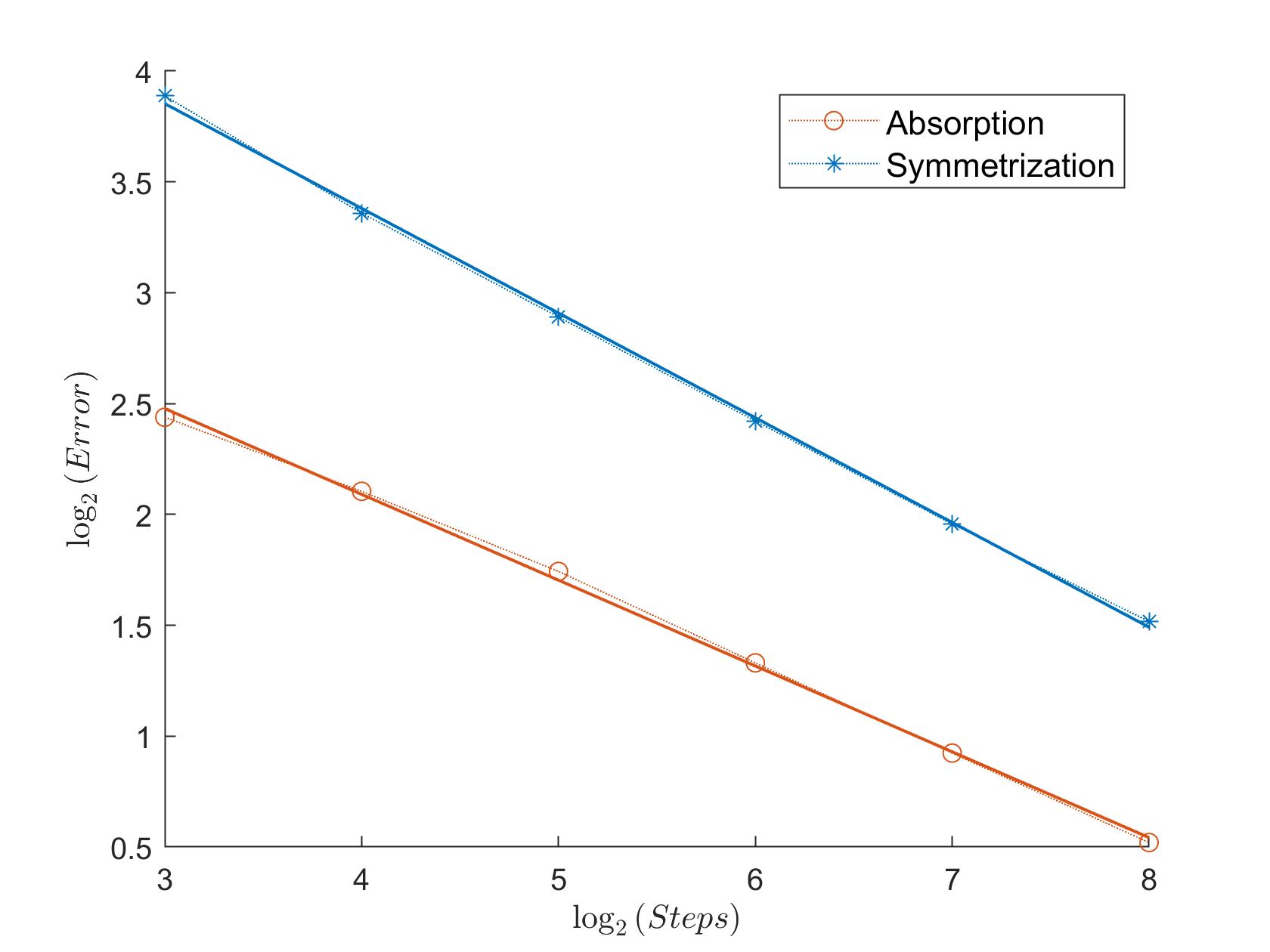}
		\caption{Call Model 4}
		\label{fig4a}
	\end{figure}
	
	Summarizing, we can confirm a  (minimum) numerical convergence rate of $min\{\nu,1\}$ for the Symmetrized and Absorbed Euler under even milder assumptions on the regularity of the payoff function. We saw even better numerical convergence results for a high Feller index. 
	\smallskip
	
	\begin{rmk} {Our analysis does not carry over to modified Euler schemes for the CIR process that take negative values, as e.g. the full-truncation Euler \cite{CS2}. 
			In these schemes  the approximation of  the  CIR  component is not bounded from below which prohibits our application of the Kolmogorov PDE and It\=o's lemma. See \cite{lord-comparison} for a survey on modified Euler schemes.}
	\end{rmk}

	\begin{rmk}\label{rem_bally}
		Bally and Talay analyze in \cite{Talay} the weak error of the Euler scheme for SDEs with $C_b^{\infty}$-coefficients, i.e. coefficients which are infinitely differentiable and whose derivatives of any order are bounded, that satisfy an additional  non-degeneracy condition of H\"ormander type (UH).  They establish  weak order one for the Euler scheme for test functions $f$ that are only measurable and bounded. However, the log-Heston model does not satisfy the above assumptions and an adaptation of the approach of \cite{Talay} to the log-Heston model leads to  {{the restrictive assumption $\nu > \frac{9}{2}$}} in \cite{AltDiss}.
		
		The cubature on Wiener space methodology, see \cite{LV}, is another approach  to deal with test functions $f$ of low smoothness. This approach transfers the quadrature of the SDE into the quadrature of a related collection of ordinary differential equations. It allows to handle  Lipschitz continuous test functions for SDEs with $C_b^{\infty}$-coefficients, which satisfy the (UHG) condition, see \cite{Kusoka}. The latter is a relaxation of the (UH) condition.
	\end{rmk}

	\section{Auxiliary Results}
	In this section, we will collect and establish several auxiliary results for the weak error analysis.

	\subsection{Kolmogorov PDE}\label{aux_pde}
	To obtain the Kolmogorov backward PDE for the Heston Model, we look at the following integral equations:
	\begin{align*}
		V^{s,v}_t&=v+\int_{s}^{t}\kappa(\theta-V^{s,v}_r)dr+\sigma\int_{s}^{t}\sqrt{V^{s,v}_r}dW_r,\\
		{X}^{s,x,v}_t &= x + \mu (t-s) -\frac{1}{2}\int_{s}^{t}V^{s,v}_r dr+\rho\int_{s}^{t}\sqrt{V^{s,v}_r}dW_r+\sqrt{1-\rho^2}\int_{s}^{t}\sqrt{V^{s,v}_r}dB_r.
	\end{align*}
	Without loss of generality, we will set $\mu=0$ throughout the remainder of this manuscript. We now define
	\begin{align*}
		u(t,x,v) &= \mathbb{E}\left[f\left({X}^{t,x,v}_T,V^{t,v}_T\right)\right], \, t\in[0,T], \, x\in \mathbb{R}, \, v\ge 0,
	\end{align*} 
	and let  $f\colon \mathbb{R} \times [0, \infty) \rightarrow \mathbb{R}$. The corresponding  Kolmogorov backward PDE is then given by
	\begin{equation}\label{kolmoPDE}
		\begin{aligned}
			u_t(t,x,v) =&\text{ }  \frac{v}{2} u_x(t,x,v)-\kappa(\theta-v)u_v(t,x,v)\\
			&-\frac{v}{2}\left(u_{xx}(t,x,v)+2\rho\sigma u_{xv}(t,x,v)+\sigma^2u_{vv}(t,x,v)\right)\text{, } \quad t\in(0,T),x\in\R,v>0,\\
			u(T,x,v) =&\text{ } f(x,v) \text{, } \quad x\in\R, v\ge0.
		\end{aligned}
	\end{equation}
	
	\smallskip
	
	In our error analysis we will follow the now classical approach of \cite{TT} which exploits the regularity of \eqref{kolmoPDE}. For the latter we will rely on the recent work of Briani et al. \cite{BCT}. To state their regularity results, we will need the following notation:
	
	\bigskip

	For a multi-index $l=(l_1,...,l_d)\in \mathbb{N}^d$, we define $s(l)=\sum_{j=1}^{d}l_j$ and for $y\in\R^d$, we define $\partial^l_y=\partial^{l_1}_{y_1}\cdot\cdot\cdot\partial^{l_d}_{y_d}$. Moreover, we denote by $|y|$ the standard Euclidean norm in $\R^d$. Let $\cD\subset\R^d$ be a domain or a closure of a domain and $q\in\mathbb{N}$. $C^q\left(\cD\right)$ is the set of all real-valued functions on $\cD$ which are $q$-times continuously differentiable. $C^q_{pol}\left(\cD\right)$ is the set of functions $g\in C^q\left(\cD\right)$ such that there exist $c,a>0$ for which
	\begin{align*}
		|\partial^l_yg(y)|\le c(1+|y|^a), \qquad y\in\cD,\, s(l)\le q.
	\end{align*}
	We define $C^q_{pol,T}\left(\cD\right)$ as the set of functions $v\in C^{\lfloor q/2\rfloor,q}_{pol}\left([0,T)\times\cD\right)$ such that there exist $c,a>0$ for which
	\begin{align*}
		\sup\limits_{t<T}|\partial^k_t\partial^l_yv(t,y)|\le c(1+|y|^a), \qquad y\in\cD, \, 2k+s(l)\le q.
	\end{align*}
	\begin{proposition}[Briani, Caramellino, Terenzi; Proposition 5.3 and Remark 5.4]\label{propbriani}
		Let $q\in\mathbb{N}$ and suppose that $f\in C^{2q}_{pol}(\R \times [0,\infty))$. Then, the solution $u$ of PDE \eqref{kolmoPDE} satisfies $u\in C^q_{pol,T}(\R \times [0,\infty))$. 
	\end{proposition}

	\smallskip
	
	Constants whose values depend only on  $T$, $x_0$, $v_0$, $\kappa$, $\theta$, $\sigma$ and $\rho$ will be denoted in the following by $C$, regardless of their value. Other dependencies will be denoted by subscripts, i.e. $C_{f,\beta}$ means that this constant depends additionally on  the function $f$ and the parameter $\beta$.
	Moreover, the value of all these constants can change from line to line.

	\smallskip
	
	\subsection{Properties of the discretization schemes}\label{propscheme}
	We write both volatility schemes as It\=o processes for $t\in[t_k,t_{k+1}]$:
	\begin{align*}
		\hat{v}^{sym}_{t}&=\biggl|v_{n(t)}^{sym}+\int_{\eta(t)}^{t}\kappa\left(\theta-v_{n(t)}^{sym}\right)ds+\sigma\int_{\eta(t)}^{t}\sqrt{v_{n(t)}^{sym}}dW_s\biggr|,\\
		\hat{v}^{abs}_{t}&=\left(v_{n(t)}^{abs}+\int_{\eta(t)}^{t}\kappa\left(\theta-v_{n(t)}^{abs}\right)ds+\sigma\int_{\eta(t)}^{t}\sqrt{v_{n(t)}^{abs}}dW_s\right)^{+}.
	\end{align*}
	Here we have set $ n(t):=\max\{k\in \{0,...,N\}:t_k\le t \}$ and $\eta(t):=t_{n(t)}$.
	The log-price discretization is the same for both schemes:
	\begin{align*}
		\hat{x}_t^{\star}&=x_{n(t)}^{\star}-\frac{1}{2}\int_{\eta(t)}^{t}v_{n(t)}^{\star}ds+\rho\int_{\eta(t)}^{t}\sqrt{v_{n(t)}^{\star}}dW_s+\sqrt{1-\rho^2}\int_{\eta(t)}^{t}\sqrt{v_{n(t)}^{\star}}dB_s,
	\end{align*}
	where $\star \in \{sym, abs\}$.
	We define
	\begin{align}\label{defZ}
		Z_t^{\star}:=v_{n(t)}^{\star}+\kappa\big(\theta-v_{n(t)}^{\star}\big)(t-\eta(t))+\sigma\sqrt{v_{n(t)}^{\star}}\left(W_{t}-W_{\eta(t)}\right)
	\end{align}
	and use the Tanaka-Meyer formulae  for $\hat{v}^{sym}_t=\bigl|Z_t^{sym}\bigr|$ and  for $\hat{v}^{abs}_t=\left(Z_t^{abs}\right)^+$ to obtain
	\begin{equation}\label{itosym}
		\begin{aligned}
			\hat{v}^{sym}_t=v_{n(t)}^{sym} & +\int_{\eta(t)}^{t} \sign\left(Z_s^{sym}\right)\kappa\left(\theta-v_{n(s)}^{sym}\right)ds+\sigma\int_{\eta(t)}^{t}\sign\left(Z_s^{sym}\right)\sqrt{v_{n(s)}^{sym}}dW_s  \\ & +\left( L_t^0(Z^{sym})-L_{\eta(t)}^0(Z^{sym})\right)
		\end{aligned}
	\end{equation}
	and
	\begin{equation}
		\begin{aligned}
			\hat{v}^{abs}_t=\text{ }v_{n(t)}^{abs} &+\int_{\eta(t)}^{t}\mathbbm{1}_{\{Z_s^{abs} > 0\}}\kappa\left(\theta-v_{n(s)}^{abs}\right)ds+\sigma\int_{\eta(t)}^{t}\mathbbm{1}_{\{Z_s^{abs}> 0\}}\sqrt{v_{n(s)}^{abs}}dW_s \\ & +\frac{1}{2}\left( L_t^0(Z^{abs})-L_{\eta(t)}^0(Z^{abs})\right). 	
		\end{aligned}
	\end{equation}
	Here $L^0(Z^{\star})=(L_t^0(Z^{\star}))_{t\in [0,T]}$ is the local time  of $Z^{\star}$ {{in zero.}}
	For almost all $\omega \in \Omega$ the map $[0,T] \ni t\mapsto [L_t^0(Z^{\star})](\omega) \in \mathbb{R}$ is continuous and non-decreasing with $L_0^0(Z^{\star})=0$. See e.g.~Theorem 7.1 in Chapter III of  \cite{KS}.
	
	Moreover, note that $$ {P}(Z_t^{sym}=0 | v_{n(t)}^{sym} =v )= {P}(Z_t^{abs}=0 | v_{n(t)}^{abs}=v )=0, \,\, \qquad v \geq 0,  \, t \in [0,T] \setminus \{t_0,t_1, \ldots, t_N\},$$ and so $P(Z_t^{\star}=0)=0$ for all $ t \in [0,T]   \setminus \{t_0,t_1, \ldots, t_N\}$ and  $\star \in \{sym, abs\}$. 
	\smallskip
	
	From \cite{BO} we know that the moments of the symmetrized Euler scheme are bounded. This also holds for the Euler scheme with absorption.
	\begin{lemma}\label{vbounded}
		For any $p\ge1$ we have that
		\begin{align*}
			\mathbb{E}\left[\sup\limits_{t\in[0,T]}\big| \hat{v}^{\star}_{t}\big|^{p}\right] < \infty
		\end{align*}
		for $\star \in \{sym, abs\}$.
	\end{lemma}
	
	\begin{proof}
		
		The proof for the symmetrized Euler scheme can be found in \cite{BO}, the proof for the absorbed Euler scheme can be done analogously.
	\end{proof}
	\smallskip
	Furthermore, the discretization of the log-price process also has  bounded moments.
	\begin{lemma}\label{xbounded}
		For any $p\ge1$ we have that
		\begin{align*}
			\mathbb{E}\left[\sup\limits_{t\in[0,T]} \big| \hat{x}_{t}^{\star} \big|^{p}\right] < \infty 
		\end{align*}
		for $\star \in \{sym, abs\}$.
	\end{lemma}
	\begin{proof}
		This follows directly from Lemma \ref{vbounded} using the H\"older and the Burkholder-Davis-Gundy inequalities.
	\end{proof}
	\smallskip
	We are now interested in the probability of $Z^{\star}$ becoming less or equal to $0$. The next lemma is similar to Lemma 3.7 in \cite{BO}{{, and we include its proof for completeness.}}
	\begin{lemma}\label{lemma3.7}
		Let $\Delta t< \frac{1}{\kappa }$. We have
		\begin{align*}
			P(Z_t^{\star}\le0)\le \mathbb{E}\left[\exp\left(-\frac{{v}_{n(t)}^{\star}(1-\kappa\Delta t)^2}{2\sigma^2\Delta t}\right)\right], \qquad t\in[0,T] \setminus \{t_0,t_1, \ldots, t_N\},
		\end{align*}
		for $\star \in \{sym, abs\}$.
	\end{lemma}
	
	\begin{proof}   Let $t\in[0,T] \setminus \{t_0,t_1, \ldots, t_N\}$.
		Since $P(Z_t^{\star}=0)=0$ we need to consider only $P(Z_t^{\star} <0)$. Throughout this proof, we will drop the $\star$-label to simplify the notation.
		By the definition of $Z$ in \eqref{defZ}, we have
		\begin{align*}
			P(Z_t < 0|{v}_{n(t)}=v)\le P\left(W_t-W_{\eta(t)}< \frac{-v(1-\kappa (t-\eta(t)))-\kappa\theta (t-\eta(t))}{\sigma\sqrt{v}}\right)
		\end{align*} for
		$v >0$ and 
		\begin{align*}
			P(Z_t < 0|{v}_{n(t)}=0)=0.
		\end{align*}
		
		For a centered Gaussian random variable $G$  with variance $\nu^2>0$, it holds that
		\begin{align*}
			P\left(G<\beta\right)\le\exp\left(-\frac{\beta^2}{2\nu^2}\right)
		\end{align*}
		for $\beta<0$. Therefore, we have
		\begin{align*}
			P(Z_t\le0)&\le\mathbb{E}\left[\exp\left(-\frac{\left({v}_{n(t)}(1-\kappa (t-\eta(t)) )+\kappa\theta (t-\eta(t))\right)^2}{2\sigma^2{v}_{n(t)} (t-\eta(t))}\right)\mathbbm{1}_{\{{v}_{n(t)}>0\}}\right]\\
			&\le\mathbb{E}\left[\exp\left(-\frac{\left({v}_{n(t)}(1-\kappa (t-\eta(t)) )\right)^2}{2\sigma^2{v}_{n(t)} (t-\eta(t))}\right)\right]\\
			&=\mathbb{E}\left[\exp\left(-\frac{{v}_{n(t)}(1-\kappa (t-\eta(t)))^2}{2\sigma^2 (t-\eta(t))} \right)\right].
		\end{align*}
		Since 
		$$ \frac{(1-\kappa (t-\eta(t)))^2}{t-\eta(t)} \geq  \frac{(1-\kappa \Delta t )^2}{\Delta t}, \qquad t \in [0,T] \setminus \{t_0,t_1, \ldots, t_N \}, $$
		the assertion follows.
		
	\end{proof}
	
	\smallskip
	
	For the further control of $P(Z_t^{\star} \leq 0)$ we will need the following technical result on a sequence that was analyzed by Cozma and Reisinger in \cite{CS2}. We are now giving a different and simplified bound which is crucial for our error analysis.
	
	\begin{lemma}\label{technical_sequence}
		Suppose that $\Delta t < \frac{1}{\kappa}$ and set
		\begin{align*}
			\alpha_N=\frac{1-\kappa\Delta t}{2}.
		\end{align*}
		(i)	Consider the sequence $(c_j)_{0\le j\le N}$ with
		\begin{align*}
			c_0=\alpha_N, \qquad c_1=\alpha_N-\alpha^2_N, \qquad c_{j+1}=c^2_j+\alpha_N-\alpha^2_N, \quad j=1, \ldots, N-1.
		\end{align*}
		Then, we have
		\begin{align*}
			c_j\le 1-\alpha_N-\frac{ \varepsilon (1-\varepsilon)}{ 1+ \varepsilon(j-1)},\qquad   j=1, \ldots, N,
		\end{align*}
		for all $\varepsilon\in(0,1/2]$. \\
		(ii)  Define the sequence $(a_j)_{0\le j\le N}$ by
		\begin{align*}
			a_j=\frac{2(\alpha_N-c_j)}{\sigma^2\Delta t}, \qquad j=0, \ldots, N.
		\end{align*}
		Then, we have $a_j \geq 0$ for $j=0, \ldots, N$.
		Moreover, let $\varepsilon \in (0, 1/2]$ and
			\begin{align} \label{const_lemma3.7} c= \exp\left(\kappa\left(\nu T+\frac{2v_0}{\sigma^2}\right)\right)\left(\max\left\{1,\frac{\sigma^2\nu}{v_0e}\right\}\right)^{\nu}. \end{align} Then, we have 
			\begin{align*}
				\exp\left(-\kappa\theta\sum_{j=0}^{k-1}a_{j+1}\Delta t\right){\exp\left(-v_0a_{k+1}\right)} \leq   c \left(\frac{\Delta t}{\varepsilon}\right)^{\nu(1-\varepsilon)}  
		\end{align*}
		for all $k=1, \ldots, N$.

	\end{lemma}
	
	\begin{proof} (i)
		Since $\Delta t\in(0,\frac{1}{\kappa})$, we know that
		\begin{align*}
			\alpha_N=\frac{1-\kappa\Delta t}{2}<\frac{1}{2}, \qquad  \alpha_N>\frac{1-\kappa\frac{1}{\kappa}}{2}=0
		\end{align*}
		and therefore $\alpha_N\in(0,1/2)$. 
		By construction, we then have $c_j \geq0$, $j=0, \ldots, N$.
		Now let $\varepsilon \in (0,1/2]$. We show that 
		\begin{align*}
			c_j\le 1-\alpha_N-\frac{1-\varepsilon}{j-1+\varepsilon^{-1}}, \qquad j=1, \ldots,  N,
		\end{align*}
		by induction. For $j=1$, we have
		\begin{align*}
			c_1=\alpha_N-\alpha^2_N=1-\alpha_N-\left(1-\alpha_N\right)^2\le 1-\alpha_N-\frac{1}{4}\le 1-\alpha_N-\frac{1-\varepsilon}{\varepsilon^{-1}},
		\end{align*}
		since $1/4  \geq (1-\varepsilon)\varepsilon$.
		
		Suppose that the statement holds for a fixed $j\in\{1,...,N\}$. Then, we have
		\begin{align*}
			c_{j+1}=\alpha_N-\alpha^2_N+c^2_j  \le &   \, \alpha_N-\alpha^2_N+\left(1-\alpha_N-\frac{1-\varepsilon}{j-1+\varepsilon^{-1}}\right)^2\\
			=& \,  \alpha_N-\alpha^2_N+1-2\alpha_N +\alpha_N^2 -2(1-\alpha_N)\frac{1-\varepsilon}{j-1+\varepsilon^{-1}}\\
			& \quad  \,\,\,  +\frac{(1-\varepsilon)^2}{(j-1+\varepsilon^{-1})^2}\\
			=& \, 1-\alpha_N-2(1-\alpha_N)\frac{1-\varepsilon}{j-1+\varepsilon^{-1}}+\frac{(1-\varepsilon)^2}{(j-1+\varepsilon^{-1})^2}.
		\end{align*}
		Now for
		$$ c_{j+1} \le 1-\alpha_N-\frac{1-\varepsilon}{j+\varepsilon^{-1}}$$
		to be true, it is sufficient that
		\begin{align} \label{final_cj}
			\frac{1-\varepsilon}{j-1+\varepsilon^{-1}}-\frac{(1-\varepsilon)^2}{(j-1+\varepsilon^{-1})^2}\ge \frac{1-\varepsilon}{j+\varepsilon^{-1}}
		\end{align}
		since $2(1-\alpha_N)\in(1,2)$. Equation \eqref{final_cj} can be verified by a simple computation.
		
		\smallskip

		\smallskip
		(ii) Since  $c_{j+1}=c^2_j+\alpha_N-\alpha^2_N$ 
		and $c_0 = \alpha_N$, $c_1 =\alpha_N - \alpha_N^2 \leq \alpha_N$, we can establish by induction that $c_j \leq \alpha_N$. Since
		$$	a_j=\frac{2(\alpha_N-c_j)}{\sigma^2\Delta t}$$ we therefore have $a_j \geq 0$ for $j=0, \ldots, N$.
		It follows that
		\begin{align*}
			-\kappa\theta\sum_{j=0}^{k-1}a_{j+1}\Delta t=\frac{2\kappa\theta}{\sigma^2}\sum_{j=0}^{k-1}\left(c_{j+1}-\alpha_N\right)&\le \frac{2\kappa\theta}{\sigma^2}\sum_{j=0}^{k-1}\left(1-2\alpha_N-\frac{1-\varepsilon}{j+\varepsilon^{-1}}\right)\\
			&\le\frac{2\kappa\theta}{\sigma^2} \int_{0}^{k}\left(1-2\alpha_N-\frac{1-\varepsilon}{j+\varepsilon^{-1}}\right)dj\\
			&=\frac{2\kappa\theta(1-\varepsilon)}{\sigma^2}\left(\ln(\varepsilon^{-1})-\ln(k+\varepsilon^{-1})\right)+\frac{2\kappa\theta}{\sigma^2}\kappa\Delta tk\\
			&\leq  \nu(1-\varepsilon) \ln\left(\frac{1}{1 + \varepsilon k}\right)+ \kappa \nu T.
		\end{align*}
		Using the definition of $a_{k+1}$ and $\alpha_N$, as well as the estimate for $c_{k+1}$ from (i) we obtain
				\begin{align*}
					\exp\left(-v_0a_{k+1}\right)&=\exp\left(\frac{2v_0}{\sigma^2\Delta t}\left(c_{k+1}-\alpha_N\right)\right)\\
					&\le\exp\left(\frac{2v_0}{\sigma^2\Delta t}\left(1-2\alpha_N-\frac{\varepsilon(1-\varepsilon)}{1+\varepsilon k}\right)\right)\\
					&\le\exp\left(\frac{2v_0\kappa}{\sigma^2}\right)\exp\left(-\frac{2v_0\varepsilon(1-\varepsilon)}{\sigma^2}\frac{1}{\Delta t}\frac{1}{1+\varepsilon k}\right)\\
					&\le\exp\left(\frac{2v_0\kappa}{\sigma^2}\right)\exp\left(-\frac{v_0\varepsilon}{\sigma^2}\frac{1}{\Delta t}\frac{1}{1+\varepsilon k}\right),
				\end{align*} 
				since we have $\varepsilon \in (0,1/2]$.
		Thus, we obtain
			\begin{align*}
				& \exp\left(-\kappa\theta\sum_{j=0}^{k-1}a_{j+1}\Delta t\right)	\exp\left(-v_0a_{k+1}\right) \\ & \qquad  \le \exp\left(\kappa\left(\nu T+\frac{2v_0}{\sigma^2}\right)\right)\exp\left(-\frac{v_0\varepsilon}{\sigma^2}\frac{1}{\Delta t}\frac{1}{1+\varepsilon k}\right)\left(\frac{1}{1+\varepsilon k}\right)^{\nu(1-\varepsilon)}
				\\ & \qquad   = \exp\left(\kappa\left(\nu T+\frac{2v_0}{\sigma^2}\right)\right)\exp\left(-\frac{v_0\varepsilon}{\sigma^2}\frac{1}{\Delta t}\frac{1}{1+\varepsilon k}\right)\left(\frac{v_0\varepsilon}{\sigma^2}\frac{1}{\Delta t}\frac{1}{1+\varepsilon k}\right)^{\nu(1-\varepsilon)} \left( \frac{\sigma^2 \Delta t }{v_0\varepsilon} \right)^{\nu(1-\varepsilon)}.
			\end{align*}
			The inequality 
			$$  x^\alpha\exp(-x)\le \alpha^\alpha\exp(-\alpha) , \qquad \alpha>0, \, x>0,$$ 	 
			and using again that $\varepsilon \in (0,1/2]$ now yield 	
			\begin{align*}
				\exp\left(-\kappa\theta\sum_{j=0}^{k-1}a_{j+1}\Delta t\right)	\exp\left(-v_0a_{k+1}\right)
				&\le\exp\left(\kappa\left(\nu T+\frac{2v_0}{\sigma^2}\right)\right) \left( \frac{\nu(1-\varepsilon)}{e}\right)^{\nu(1-\varepsilon)} \left(\frac{\sigma^2 \Delta t }{v_0\varepsilon}\right)^{\nu(1-\varepsilon)} 
				\\	&\le \exp\left(\kappa\left(\nu T+\frac{2v_0}{\sigma^2}\right)\right)\left(\frac{\sigma^2\nu}{v_0e}\right)^{\nu(1-\varepsilon)} \left(\frac{\Delta t}{\varepsilon}\right)^{\nu(1-\varepsilon)}\\
				&\le\exp\left(\kappa\left(\nu T+\frac{2v_0}{\sigma^2}\right)\right)\left(\max\left\{1,\frac{\sigma^2\nu}{v_0e}\right\}\right)^{\nu}\left(\frac{\Delta t}{\varepsilon}\right)^{\nu(1-\varepsilon)},
			\end{align*}
			which finishes the proof. 
	\end{proof}

	\smallskip

	\smallskip
	The next Lemma  gives an upper bound for the expression from Lemma \ref{lemma3.7}. It plays the same role  as  Lemma 3.6 in \cite{BO}  and  in comparison to this Lemma it removes the restriction on $\nu$ and also obtains a better estimate in terms of $\nu$ for $P(Z_t^{\star} \leq 0) $.
	
	\begin{proposition}\label{lemma3.6} 
		For $\Delta t < \frac{1}{\kappa}$ and $\varepsilon\in(0,1/2]$ we have that
		\begin{align}\label{termlemma3.6}
			\mathbb{E}\left[\exp\left(-\frac{{v}_{k}^{\star}(1-\kappa\Delta t)^2}{2\sigma^2\Delta t}\right)\right]\le c \left(\frac{\Delta t}{\varepsilon}\right)^{\nu(1-\varepsilon)}, \qquad k=0, \ldots, N,
		\end{align}
		and
		\begin{align}\label{problemma3.6}
			P(Z_t^{\star} \leq 0) \leq   c \left(\frac{\Delta t}{\varepsilon}\right)^{\nu(1-\varepsilon)}, \qquad t \in [0,T]  \setminus \{t_0,t_1, \ldots, t_N\},
		\end{align}
		for $\star \in \{sym, abs\}$, where $c$ is given by   \eqref{const_lemma3.7}.
	\end{proposition}
	\begin{proof}
		Lemma \ref{lemma3.7} and \eqref{termlemma3.6} directly give \eqref{problemma3.6}. So it remains to show  \eqref{termlemma3.6}.

		The first step of this proof is to describe a sequence $(a_j)_{0\le j\le N}$ whose first element is equal to $\frac{(1-\kappa\Delta t)^2}{2\sigma^2\Delta t}$ and which has some suitable properties to bound the term on the left side of \eqref{termlemma3.6}. Suppose that $\Delta t < \frac{1}{\kappa}$. Define  the sequence $(a_j)_{0 \leq j \leq N}$ as in the previous Lemma, i.e.
		\begin{align*}
			a_j=\frac{2(\alpha_N-c_j)}{\sigma^2\Delta t}
		\end{align*}
		with
		\begin{align*}
			c_0=\alpha_N, \qquad c_1=\alpha_N-\alpha^2_N, \qquad c_{j+1}=c^2_j+\alpha_N-\alpha^2_N, \quad j=1, \ldots, N-1,
		\end{align*}
		and $\alpha_N = \frac{1-\kappa\Delta t}{2}$. In particular, we have $a_0=0,$ $$a_1=\frac{2\alpha^2_N}{\sigma^2\Delta t}=\frac{(1-\kappa\Delta t)^2}{2\sigma^2\Delta t} $$
		and 
		\begin{align*}
			a_{j+1}=\frac{2(\alpha_N-c_{j+1})}{\sigma^2\Delta t}=\frac{2(\alpha^2_N-c^2_j)}{\sigma^2\Delta t}=\frac{4\alpha_N(\alpha_N-c_j)-2(\alpha_N-c_j)^2}{\sigma^2\Delta t}=2\alpha_Na_j-\frac{1}{2}a^2_j\sigma^2\Delta t.
		\end{align*}
		
		Next, we take a look at 
		\begin{align*}
			\mathbb{E}\left[\exp\left(-{v}_{k}^{\star}a_i\right)\right]=\mathbb{E}\left[\mathbb{E}\left[\exp\left(-{v}_{k}^{\star}a_i\right)\big|\cF_{t_{k-1}}\right]\right]
		\end{align*}
		and bound the conditional expectation using that $a_i \geq 0$, $|v| \geq v$ and  $v^+ \geq v$, respectively. We have
		\begin{align*}
			&\mathbb{E}\left[\exp\left(-{v}_{k}^{\star}a_i\right)\big|\cF_{t_{k-1}}\right] \\ & \quad \le\mathbb{E}\left[\exp\left(-a_i\left(\kappa\theta\Delta t+v_{{k-1}}^{\star}(1-\kappa\Delta t)+\sigma\sqrt{{v}_{{k-1}}^{\star}}\left(W_{t_k}-W_{t_{k-1}}\right) \right)\right)\big|\cF_{t_{k-1}}\right]\\
			& \quad=\exp\left(-a_i\left(\kappa\theta\Delta t+{v}_{{k-1}}^{\star}(1-\kappa\Delta t)\right)\right)\mathbb{E}\left[\exp\left(-a_i\sigma\sqrt{{v}_{{k-1}}^{\star}}\left(W_{t_k}-W_{t_{k-1}}\right)\right)\big|\cF_{t_{k-1}}\right]\\
			& \quad = \exp\left(-a_i\left(\kappa\theta\Delta t+{v}_{{k-1}}^{\star}(1-\kappa\Delta t)\right)\right)\exp\left(\frac{1}{2}a^2_i\sigma^2{v}_{{k-1}}^{\star}\Delta t\right).
		\end{align*}
		Since
		\begin{align*}
			a_{i+1}=a_i(1-\kappa \Delta t)-\frac{1}{2}a^2_i\sigma^2\Delta t,
		\end{align*}
		it follows
		\begin{align*}
			\mathbb{E}\left[\exp\left(-{v}_{k}^{\star}a_i\right)\right]\le\exp\left(-a_i\kappa\theta\Delta t\right)\mathbb{E}\left[\exp\left(-{v}_{{k-1}}^{\star}a_{i+1}\right)\right].
		\end{align*}
		Plugging in $a_1$ and applying this upper bound $k$ times, we arrive at
		\begin{align*}
			\mathbb{E}\left[\exp\left(-\frac{{v}_{k}^{\star}(1-\kappa\Delta t)^2}{2\sigma^2\Delta t}\right)\right]=\mathbb{E}\left[\exp\left(-{v}_{k}^{\star}a_1\right)\right]\le\exp\left(-\kappa\theta\sum_{j=1}^{k}a_{j}\Delta t\right)\exp\left(-v_0a_{k+1}\right).
		\end{align*}
		The assertion now follows from the second part of the previous Lemma. 
	\end{proof}
	\smallskip

	We also need the following two $L^p$-results.
	\begin{lemma}\label{localsym}
		For all $p\ge1$ there exists a constant $C_{p}>0$  such that 
		\begin{align*}
			\sup\limits_{t\in[0,T]}\mathbb{E}\left[\big|\hat{v}_t^{\star}-\hat{v}_{\eta(t)}^{\star}\big|^{p}\right]\le C_{p}\left(\Delta t\right)^{\frac{p}{2}}
		\end{align*} holds
		for  $\star \in \{sym, abs\}$. 
		\begin{proof}
			The proof for the Symmetrized Euler can be found in \cite{BER}. We prove the statement for the Absorbed Euler by a similar approach and drop  the $abs$-label to simplify the notation.  We have
			\begin{align*}
				\sup\limits_{t\in[0,T]}\mathbb{E}\left[\left|\hat{v}_t-\hat{v}_{\eta(t)}\right|^{p}\right]= \sup\limits_{k=0,...,N-1}\sup\limits_{t\in[t_k,t_{k+1}]}\mathbb{E}\left[\left|\hat{v}_t-\hat{v}_{t_k}\right|^{p}\right].
			\end{align*}
			Since 
					\begin{align*}
						\left|\left(v+z\right)^+-v \right| \leq |z| 
					\end{align*}
					for $v>0$ and $z\in\mathbb{R}$, we have for $t\in[t_k,t_{k+1}]$ that
			\begin{align*}
				\left|\hat{v}_t-\hat{v}_{t_k}\right|^{p}&=\left|\left(\hat{v}_{t_k}+\kappa\left(\theta-\hat{v}_{t_k}\right)(t-t_k)+\sigma\sqrt{\hat{v}_{t_k}}\left(W_{t}-W_{t_k}\right)\right)^+-\hat{v}_{t_k}\right|^p \\
				&\le\left|\kappa\left(\theta-\hat{v}_{t_k}\right)(t-t_k)+\sigma\sqrt{\hat{v}_{t_k}}\left(W_{t}-W_{t_k}\right)\right|^{p}\\
				&\le2^{p-1}\left(\Delta t\right)^{p}\left|\kappa\left(\theta-\hat{v}_{t_k}\right)\right|^{p}+2^{p-1}\sigma^{p}\hat{v}_{t_k}^{p/2}\left|W_{t}-W_{t_k}\right|^{p}.
			\end{align*}
			Using Lemma \ref{vbounded} we obtain for all $k =0, \ldots, N-1$ that
			\begin{align*}
				\mathbb{E}\left[\left|\kappa\left(\theta-\hat{v}_{t_k}\right)\right|^{p}\right]\le C_p
			\end{align*}
			and
			\begin{align*}
				\mathbb{E}\left[\hat{v}_{t_k}^{p/2}\left|W_{t}-W_{t_k}\right|^{p}\right]\le \left( \mathbb{E}\left[\hat{v}_{t_k}^{p}\right] \right)^{\frac{1}{2}} \left( \mathbb{E}\left[\left|W_{t}-W_{t_k}\right|^{2p}\right] \right)^{\frac{1}{2}}
				&\le C_p\left(\Delta t\right)^{\frac{p}{2}}.
			\end{align*}
			Therefore we have
			\begin{align*}
				\sup\limits_{k=0,...,N-1}	\sup\limits_{t\in[t_k,t_{k+1}]}\mathbb{E}\left[\left|\hat{v}_t-\hat{v}_{t_k}\right|^{p}\right]\le C_p\left(\Delta t\right)^{\frac{p}{2}}
			\end{align*}
			and the statement follows.
		\end{proof}
	\end{lemma}
	
	\smallskip
	By a similar calculation and with Lemma \ref{vbounded}, we obtain the following lemma for the discretization of the price process.
	
	\begin{lemma}\label{localxsymm}
		For all $p\ge1$ there exists a constant $C_{p}>0$ such that
		\begin{align*}
			\sup\limits_{t\in[0,T]}\mathbb{E}\left[\big|\hat{x}_t^{\star}-\hat{x}^{\star}_{\eta(t)}\big|^{p}\right]\le C_{p}\left(\Delta t\right)^{\frac{p}{2}}
		\end{align*}
		holds for  $\star \in \{sym, abs\}$.
	\end{lemma}

	\subsection{Properties of the local time}
	We need an upper bound for the expected local time of $Z^{\star}$ in $0$. Our proof follows similar ideas as the proof of Proposition 3.5 in \cite{BO},  but adds the results from Section \ref{propscheme}.
	\begin{proposition}\label{proplocaltime} 
		Let $\beta>0$, $\delta>0$, $\varepsilon \in (0,1/2]$, $\Delta t\le\frac{1}{2\kappa}$ and $\star \in \{sym, abs\}$. Then, there exist constants $C_{\delta}>0$ and $C_{\beta,\delta}>0$  such that
		\begin{align*}
			\mathbb{E}\left[L_t^0\left(Z^{\star}\right)-L_{\eta(t)}^0\left(Z^{\star}\right)\right]\le C_\delta\Delta t \left(\frac{\Delta t}{\varepsilon}\right)^{\nu\frac{1-\varepsilon}{1+\delta}}, \qquad t \in [0,T],
		\end{align*}
		and
		\begin{align*}
			\mathbb{E}\left[\left|L_t^0\left(Z^{\star}\right)-L_{\eta(t)}^0\left(Z^{\star}\right)\right|^{1+\beta}\right]^{\frac{1}{1+\beta}}\le C_{\beta,\delta}\left(\Delta t\right)^{\frac{1}{(1+\beta)^2}}\left(\frac{\Delta t}{\varepsilon}\right)^{\nu\frac{1-\varepsilon}{(1+\delta)(1+\beta)^2}} \qquad t \in [0,T].
		\end{align*}
		\begin{proof} (i) To simplify the notation, we drop the $\star$-label.
			By the occupation time formula, see e.g.~Theorem 7.1 in Chapter III of  \cite{KS}, we have for any $t\in[t_k,t_{k+1}]$ and for any non-negative Borel-measurable function $\phi : \mathbb{R} \rightarrow \mathbb{R}$ that $P$-a.s
			\begin{align*}
				\int_{\mathbb{R}}\phi(z)\left(L_t^z\left(Z\right)-L^z_{t_k}\left(Z\right)\right)dz&=\int_{t_k}^{t}\phi(Z_s)d\langle Z\rangle_s=\sigma^2\int_{t_k}^{t}\phi(Z_s)\hat{v}_{t_k}ds.
			\end{align*}
			Here $L^z(Z)$ is the local time of $Z$ in $z \in \mathbb{R}$.
			Since
			$$ P^{Z_s | \hat{v}_{\eta(s)}=v }= \mathcal{N} \left(v +  \kappa(\theta -v) (s-\eta(s)), \sigma^2 v (s- \eta(s)) \right)$$
			we have 
			for any $v > 0$ that
			\begin{align*}
				&\int_{\mathbb{R}}\phi(z)\mathbb{E}\left[L_t^z\left(Z\right)-L^z_{t_k}\left(Z\right)|\hat{v}_{t_k}=v\right]dz=\sigma^2\int_{t_k}^{t}v\mathbb{E}\left[\phi(Z_s)|\hat{v}_{t_k}=v\right]ds\\
				&\text{ }=\sigma\int_{\mathbb{R}}\phi(z)\int_{t_k}^{t}\frac{\sqrt{v}}{\sqrt{2\pi(s-t_k)}}\exp\left(-\frac{\left(z-v-\kappa(\theta-v)(s-t_k)\right)^2}{2\sigma^2v(s-t_k)}\right)dsdz.
			\end{align*}
			Since the above equation holds  for any non-negative Borel-measurable function $\phi$, we must have that
			\begin{align*}
				\mathbb{E}\left[L_t^z\left(Z\right)-L^z_{t_k}\left(Z\right)|\hat{v}_{t_k}=v\right]&=\sigma\int_{t_k}^{t}\frac{\sqrt{v}}{\sqrt{2\pi(s-t_k)}}\exp\left(-\frac{\left(z-v-\kappa(\theta-v)(s-t_k)\right)^2}{2\sigma^2v(s-t_k)}\right)ds
			\end{align*}
			for any $z\in\mathbb{R}$.
			Setting $z=0$ yields
			\begin{align*}
				\mathbb{E}\left[L_t^0\left(Z\right)-L_{t_k}^0\left(Z\right)|\hat{v}_{t_k}=v\right]&=\sigma\int_{t_k}^{t}\frac{\sqrt{v}}{\sqrt{2\pi (s-t_k)}}\exp\left(-\frac{\left(v+\kappa(\theta-v)\left(s-t_k\right)\right)^2}{2\sigma^2v\left(s-t_k\right)}\right)ds\\
				&\le \sigma\int_{t_k}^{t}\frac{\sqrt{v}}{\sqrt{2\pi \left(s-t_k\right)}}\exp\left(-\frac{v\left(1-\kappa \left(s-t_k\right)\right)^2}{2\sigma^2\left(s-t_k\right)}\right)ds.
			\end{align*}
			
			Since for $\delta>0$ there exist a $c_{\delta}>0$ such that $\sqrt{y}\exp(-y)\le c_{\delta}\exp\left(-\frac{y}{1+\delta}\right)$ for all $y \geq 0$, we have
			$$ \frac{\sqrt{ v}(1-\kappa  \left(s-t_k\right))}{\sqrt{2 \sigma^2  \left(s-t_k\right) }}\exp\left(-\frac{{v}\left(1-\kappa  \left(s-t_k\right) \right)^2}{2\sigma^2  \left(s-t_k\right)}\right) \leq c_{\delta} \exp\left(-\frac{{v}\left(1-\kappa  \left(s-t_k\right) \right)^2}{2\sigma^2  \left(s-t_k\right)(1+\delta) }\right).$$
			Moreover, since $1-\kappa  \left(s-t_k\right) \in [1/2,1]$
			we obtain
			$$ \frac{\sqrt{ v}}{\sqrt{s-t_k}}\exp\left(-\frac{{v}\left(1-\kappa  \left(s-t_k\right)\right)^2}{2\sigma^2\left(s-t_k\right)}\right) \leq \sqrt{8} \sigma c_{\delta}   \exp\left(-\frac{{v}\left(1-\kappa  \left(s-t_k\right) \right)^2}{2\sigma^2\left(s-t_k\right)(1+\delta)}\right).$$

			It follows
			\begin{align*}
				\mathbb{E}\left[L_t^0\left(Z\right)-L_{t_k}^0\left(Z\right)|\cF_{t_k}\right]&\le \sigma\int_{t_k}^{t}\frac{\sqrt{\hat{v}_{t_k}}}{\sqrt{2\pi \left(s-t_k\right)}}\exp\left(-\frac{ \hat{v}_{t_k}\left(1-\kappa (s-t_k)\right)^2}{2\sigma^2\left(s-t_k\right)}\right)ds\\
				&\le c_{\delta} \frac{2 \sigma^2}{\sqrt{\pi}} \int_{t_k}^{t}\exp\left(-\frac{\hat{v}_{t_k}\left(1-\kappa (s-t_k) \right)^2}{2\sigma^2(s-t_k)(1+\delta)}\right)ds\\
				&\le c_{\delta}   \frac{2 \sigma^2}{\sqrt{\pi}}  \int_{t_k}^{t} \exp\left(-\frac{\hat{v}_{t_k}\left(1-\kappa \Delta t\right)^2}{2\sigma^2\Delta t(1+\delta)}\right)ds.
			\end{align*}
			Now,  the Lyapunov inequality and Proposition \ref{lemma3.6} yield
			\begin{align*}
				\mathbb{E}\left[L_t^0\left(Z\right)-L_{t_k}^0\left(Z\right)\right]&\le C_{\delta} \int_{t_k}^{t}\mathbb{E}\left[\exp\left(-\frac{\hat{v}_{t_k}\left(1-\kappa \Delta t\right)^2}{2\sigma^2\Delta t(1+\delta)}\right)\right]ds\\
				&=C_{\delta} \int_{t_k}^{t}\mathbb{E}\left[\exp\left(-\frac{\hat{v}_{t_k}\left(1-\kappa \Delta t\right)^2}{2\sigma^2\Delta t}\right)^{\frac{1}{1+\delta}}\right]ds\\
				&\le C_{\delta} \int_{t_k}^{t}\left( \mathbb{E}\left[\exp\left(-\frac{\hat{v}_{t_k}\left(1-\kappa \Delta t\right)^2}{2\sigma^2\Delta t}\right)\right] \right)^{\frac{1}{1+\delta}}ds\\
				&\le C_\delta\Delta t \left(\frac{\Delta t}{\varepsilon}\right)^{\nu\frac{1-\varepsilon}{1+\delta}}.
			\end{align*}
			(ii) For the second statement note first that
				\begin{align*}
					& 	\mathbb{E}\left[\left| L^0_{t}(Z)-L^0_{\eta(t)}(Z)\right|^{1+\beta}\right]^{\frac{1}{1+\beta}} \\ & \qquad \qquad =\mathbb{E}\left[\left(L^0_{t}(Z)-L^0_{\eta(t)}(Z)\right)^{\frac{1}{1+\beta}}\left(L^0_{t}(Z)-L^0_{\eta(t)}(Z)\right)^{\beta+1-\frac{1}{1+\beta}}\right]^{\frac{1}{1+\beta}}\\
					& \qquad \qquad \le \left( \mathbb{E}\left[L^0_{t}(Z)-L^0_{\eta(t)}(Z)\right]\right)^{\frac{1}{(1+\beta)^2}} \left( \mathbb{E}\left[\left(L^0_{t}(Z)-L^0_{\eta(t)}(Z)\right)^{\frac{(\beta+1)^2-1}{\beta}}\right] \right)^{\frac{\beta}{(1+\beta)^2}}
				\end{align*} by H\"older's inequality.
			Now, consider first $Z=Z^{sym}$ and note that
			\begin{align*}
				&\mathbb{E}\left[\left|L_t^0\left(Z\right)-L_{\eta(t)}^0\left(Z\right)\right|^p\right]\\
				& \quad = \mathbb{E}\left[\left|\hat{v}_t-\hat{v}_{\eta(t)}-\int_{\eta(t)}^{t}\sign\left(Z_s\right)\kappa\left(\theta-\hat{v}_{\eta(s)}\right)ds-\sigma\int_{\eta(t)}^{t}\sign\left(Z_s\right)\sqrt{\hat{v}_{\eta(s)}}dW_s\right|^p\right]\\
				&  \quad \leq    3^{p-1}\left(\mathbb{E}\left| \hat{v}_t - \hat{v}_{\eta(t)} \right|^p+\mathbb{E}\left[\left|\int_{\eta(t)}^{t}\sign\left(Z_s\right)\kappa\left(\theta-\hat{v}_{\eta(s)}\right)ds\right|^p\right]\right.\\
				& \qquad \qquad  \quad \qquad \qquad  \qquad \left. \, + \, \mathbb{E}\left[\left|\sigma\int_{\eta(t)}^{t} \sign\left(Z_s\right)\sqrt{\hat{v}_{\eta(s)}}dW_s\right|^p\right]\right)
			\end{align*}
			since $|x+y+z|^p\leq 3^{p-1}(|x|^p+|y|^p+|z|^p)$ for $ x,y,z \in \mathbb{R}$, $p \geq 1$.
			We can conclude from Lemma \ref{localsym}, the H\"older inequality, the Burkholder-Davis-Gundy inequality and Lemma  \ref{vbounded} that $$ \sup_{t \in [0,T]} \mathbb{E}\left[\left|L_t^0\left(Z\right)-L_{\eta(t)}^0\left(Z\right)\right|^p\right] < \infty .$$ The case $Z=Z^{abs}$ can be done analogously. Applying the estimate from the first part, we obtain
				\begin{align*}
					\mathbb{E}\left[\left| L^0_{t}(Z)-L^0_{\eta(t)}(Z)\right|^{1+\beta}\right]^{\frac{1}{1+\beta}}&\le C_\beta \left( \mathbb{E}\left[L^0_{t}(Z)-L^0_{\eta(t)}(Z)\right]\right)^{\frac{1}{(1+\beta)^2}}\\
					&\le C_{\beta,\delta}\left(\Delta t\right)^{\frac{1}{(1+\beta)^2}} \left(\frac{\Delta t}{\varepsilon}\right)^{\nu\frac{1-\varepsilon}{(1+\delta)(1+\beta)^2}}.
				\end{align*}
		\end{proof}
	\end{proposition}
	
	\smallskip
	\section{Proof of Theorem \ref{maintheorem}}
	Now, we have enough tools to prove our main statement. 
	Since $\EE\left[u(T,x_N^{\star},v_N^{\star})\right]= \EE \left[f(x_N^{\star},v_N^{\star})\right]$ and $u(0,x_0,v_0)= \EE \left[f(X_T,V_T)\right]$
	the weak error is a telescoping sum of local errors:
	\begin{align*}
		\left|\mathbb{E}\left[f(x_N^{\star},v_N^{\star})\right]-\mathbb{E}\left[f(X_T,V_T)\right]\right|=\left|\sum_{n=1}^{N}\mathbb{E}\left[u(t_n,x_n^{\star},v_n^{\star})-u(t_{n-1},x_{n-1}^{\star},v_{n-1}^{\star})\right]\right|.
	\end{align*}
	Since $\hat{v}_t^{abs}$ can be zero with positive probability on a non-empty time interval, technical difficulties with a direct application of the It\=o-formula to $u$ at $v=0$, i.e. at the boundary of the state space, arise.
	Therefore, 
	we will analyze first
	$$  \left|\sum_{n=1}^{N}\mathbb{E}\left[u(t_n,x_n^{\star},v_n^{\star}+\gamma)-u(t_{n-1},x_{n-1}^{\star},v_{n-1}^{\star}+\gamma)\right]\right|$$
	with $\gamma >0$ and in a second step exploit that
	\begin{align}\label{gamma_reg}
		|\EE \left[f(x_N,v_N)- f(X_T,V_T)\right]|=\lim_{\gamma \searrow 0}
		\left|\sum_{n=1}^{N}\mathbb{E}\left[u(t_n,x_n^{\star},v_n^{\star}+\gamma)-u(t_{n-1},x_{n-1}^{\star},v_{n-1}^{\star}+\gamma)\right]\right|. \end{align}
	This regularization is not required for the symmetrized Euler scheme, but to present both proofs in a concise way, we use it for both schemes.
	Equation \eqref{gamma_reg} follows e.g. from the smoothness of the Kolmogorov PDE, i.e. Proposition \ref{propbriani}, and the moment bounds for the discretization schemes, i.e. Lemma \ref{vbounded} and \ref{xbounded}.
	
	\smallskip
	
	Recall that the Kolmogorov PDE is given by
	$$\begin{aligned}
		u_t(t,x,v) =&\text{ }\mathcal{L}(u)(t,x,v)\text{, } \quad t\in(0,T),x\in\R,v>0,\\
		u(T,x,v) =&\text{ } f(x,v) \text{, } \quad \qquad  x\in\R, v\ge0,
	\end{aligned}$$
	with
	$$ \mathcal{L}(u)(t,x,v)= \frac{v}{2} u_x(t,x,v)-\kappa(\theta-v)u_v(t,x,v)
	-\frac{v}{2}\left(u_{xx}(t,x,v)+2\rho\sigma u_{xv}(t,x,v)+\sigma^2u_{vv}(t,x,v)\right).
	$$ 
	For $\gamma \geq 0$, define  $u^{\gamma}(t,x,v)=u(t,x,v+\gamma)$. A direct application of Proposition \ref{propbriani} gives:
	\begin{lemma}\label{lembriani}
		Let $f\in C^{6}_{pol}(\R \times [0,\infty))$. Then, there exist $C_f>0$ and $a>0$ such that
		\begin{align*}
			\sup_{\gamma \in [0,1]}	\sup\limits_{t \in [0,T)}|\partial^l_x \partial^m_v u^{\gamma}(t,x,v)|\le C_f(1+|x|^a + |1+v|^a ), \text{ }\text{ }x \in \mathbb{R}, v > 0,
		\end{align*} if  $l+m\le 3$.
	\end{lemma}

	Moreover,  note that the function $u^{\gamma}$ satisfies by construction the PDE
	$$\begin{aligned}
		u_t^{\gamma}(t,x,v) =&\text{ }  \mathcal{L}(u^{\gamma})(t,x,v) + \frac{\gamma}{2} \mathcal{R}(u^{\gamma})(t,x,v) , \\
		u^{\gamma}(T,x,v) =&\text{ } f(x,v+\gamma) ,
	\end{aligned}$$
	where 
	\begin{align*}
		\mathcal{R}(u^{\gamma})(t,x,v)=u_x^{\gamma}(t,x,v)+2\kappa  u_v^{\gamma}(t,x,v)  - u_{xx}^{\gamma}(t,x,v)-2\rho\sigma u_{xv}^{\gamma}(t,x,v)-\sigma^2u_{vv}^{\gamma}(t,x,v). 
	\end{align*}
	Lemma \ref{lembriani} yields
	\begin{align}
		\label{bound_rem_pde}
		& \sup_{\gamma \in [0,1]}	\sup\limits_{t \in [0,T)}|\mathcal{R}(u^{\gamma})(t,x,v)|\le C_{f}(1+|x|^a + |1+v|^a ), \text{ }\text{ }x \in \mathbb{R}, v > 0, \\ & \label{bound_rem_pde-2}
		\sup_{\gamma \in [0,1]}	\sup\limits_{t \in [0,T)} \left| \frac{\partial}{\partial x} \mathcal{R}(u^{\gamma})(t,x,v) \right|\le C_{f}(1+|x|^a + |1+v|^a ), \text{ }\text{ }x \in \mathbb{R}, v > 0, \\ & \label{bound_rem_pde-3}
		\sup_{\gamma \in [0,1]}	\sup\limits_{t \in [0,T)} \left| \frac{\partial}{\partial v} \mathcal{R}(u^{\gamma})(t,x,v) \right|\le C_{f}(1+|x|^a + |1+v|^a ), \text{ }\text{ }x \in \mathbb{R}, v > 0.
	\end{align}

	We will start with the symmetrized Euler scheme and we will see then that the same procedure can be directly applied to the absorbed Euler scheme.
	
	\subsection{The symmetrized Euler scheme}\label{proofsym}
	We again drop  the $sym$-label to simplify the notation. After the previous preparations, we now apply the It\=o formula with $\gamma \in (0,1]$ to the summands of the telescoping sum. Using \eqref{disc-heston} and \eqref{itosym}
	we have
	\begin{align*}
		e_n^{\gamma} :=&\text{ }\mathbb{E}\left[u^{\gamma}(t_{n+1},x_{n+1},v_{n+1})-u^{\gamma} (t_{n},x_{n},v_n)\right]\\
		=&\int_{t_n}^{t_{n+1}}\mathbb{E}\left[u_t^{\gamma}(t,\hat{x}_t,\hat{v}_t)-\frac{1}{2}\hat{v}_{\eta(t)}u_x^{\gamma}(t,\hat{x}_t,\hat{v}_t)+\kappa(\theta-\hat{v}_{\eta(t)})\sign\left(Z_t\right)u_v^{\gamma}(t,\hat{x}_t,\hat{v}_t)\right.\\
		&\left. \qquad \qquad +\frac{1}{2}\hat{v}_{\eta(t)}u_{xx}^{\gamma}(t,\hat{x}_t,\hat{v}_t)+\rho\sigma \hat{v}_{\eta(t)} \sign\left(Z_t\right)u_{xv}^{\gamma}(t,\hat{x}_t,\hat{v}_t)+\frac{1}{2}\sigma^2 \hat{v}_{\eta(t)}u_{vv}^{\gamma}(t,\hat{x}_t,\hat{v}_t) \right]dt\\
		& \quad + \mathbb{E}\left[\int_{t_n}^{t_{n+1}}u_v^{\gamma}(t,\hat{x}_t,\hat{v}_t)dL_t^0(Z)\right] \\
		=&\int_{t_n}^{t_{n+1}}\mathbb{E}\left[u_t^{\gamma}(t,\hat{x}_t,\hat{v}_t)-\frac{1}{2}\hat{v}_{\eta(t)}u_x^{\gamma}(t,\hat{x}_t,\hat{v}_t)+\kappa(\theta-\hat{v}_{\eta(t)}) u_v^{\gamma}(t,\hat{x}_t,\hat{v}_t)\right.\\
		&\left. \qquad \qquad +\frac{1}{2}\hat{v}_{\eta(t)}u_{xx}^{\gamma}(t,\hat{x}_t,\hat{v}_t)+\rho\sigma\hat{v}_{\eta(t)}u_{xv}^{\gamma}(t,\hat{x}_t,\hat{v}_t)+\frac{1}{2} \sigma^2 \hat{v}_{\eta(t)}u_{vv}^{\gamma}(t,\hat{x}_t,\hat{v}_t)\right]dt\\
		& \quad  - 2  \int_{t_n}^{t_{n+1}} \mathbb{E}\left[ \mathbbm{1}_{\{Z_t\le0\}} \left( \kappa(\theta-\hat{v}_{\eta(t)})u_v^{\gamma}(t,\hat{x}_t,\hat{v}_t) + \rho\sigma \hat{v}_{\eta(t)}u_{xv}^{\gamma}(t,\hat{x}_t,\hat{v}_t)  \right)  \right] dt  \\
		& \quad + \mathbb{E}\left[\int_{t_n}^{t_{n+1}}u_v^{\gamma}(t,\hat{x}_t,\hat{v}_t)dL_t^0(Z)\right].
	\end{align*}
	Note that $t \mapsto L_t^0(Z)$ is pathwise increasing and that $\int_{t_n}^{t_{n+1}}u_v^{\gamma}(t,\hat{x}_t,\hat{v}_t)dL_t^0(Z)$ is a pathwise Riemann-Stieltjes integral.
	
	Since 
	$$ u_t^{\gamma}(t,x,v) =\text{ }  \mathcal{L}(u^{\gamma})(t,x,v) + \frac{\gamma}{2} \mathcal{R}(u^{\gamma})(t,x,v)$$
	with
	$$ \mathcal{L}(u^{\gamma})(t,x,v)= \frac{v}{2} u^{\gamma}_x(t,x,v)-\kappa(\theta-v)u^{\gamma}_v(t,x,v)
	-\frac{v}{2}\left(u_{xx}^{\gamma}(t,x,v)+2\rho\sigma u_{xv}^{\gamma}(t,x,v)+\sigma^2u_{vv}^{\gamma}(t,x,v)\right)
	$$ 
	and
	\begin{align*}
		\mathcal{R}(u^{\gamma})(t,x,v)=u_x^{\gamma}(t,x,v)+2 \kappa  u_v^{\gamma}(t,x,v)  - u_{xx}^{\gamma}(t,x,v)-2\rho\sigma u_{xv}^{\gamma}(t,x,v)-\sigma^2u_{vv}^{\gamma}(t,x,v)
	\end{align*}
	we can write
	\begin{align*}
		e_n^{\gamma} =&\int_{t_n}^{t_{n+1}}\mathbb{E}\left[ \frac{\hat{v}_t-\hat{v}_{\eta(t)}}{2} \left(u^{\gamma}_x(t,\hat{x}_t,\hat{v}_t)+ 2 \kappa u^{\gamma}_v(t,\hat{x}_t,\hat{v}_t) - u_{xx}^{\gamma}(t,\hat{x}_t,\hat{v}_t)- 2 \rho\sigma u_{xv}^{\gamma}(t,\hat{x}_t,\hat{v}_t)-\sigma^2 u^{\gamma}_{vv}(t,\hat{x}_t,\hat{v}_t) \right)\right]dt\\
		\quad& +\frac{\gamma}{2}\int_{t_n}^{t_{n+1}}
		\mathbb{E}\left[\mathcal{R}(u^{\gamma})(t,\hat{x}_t,\hat{v}_t)\right]dt	
		\\ \quad &- 2   \int_{t_n}^{t_{n+1}} \mathbb{E}\left[ \mathbbm{1}_{\{Z_t\le0\}} \left( \kappa(\theta-\hat{v}_{\eta(t)})u_v^{\gamma}(t,\hat{x}_t,\hat{v}_t) + \rho\sigma \hat{v}_{\eta(t)}u_{xv}^{\gamma}(t,\hat{x}_t,\hat{v}_t)  \right)  \right] dt
		\\ \quad &+\mathbb{E}\left[\int_{t_n}^{t_{n+1}}u_v^{\gamma}(t,\hat{x}_t,\hat{v}_t)dL_t^0(Z)\right] \\ 
		=&\quad e^{(1,\gamma)}_n+  e^{(2,\gamma)}_n+ e^{(3,\gamma)}_n+ e^{(4,\gamma)}_n
	\end{align*}
	with
	\begin{align*}
		e^{(1,\gamma)}_n&:=\mathbb{E}\left[\int_{t_n}^{t_{n+1}}u^{\gamma}_v(t,\hat{x}_t,\hat{v}_t)dL_t^0(Z)\right],\\
		e^{(2,\gamma)}_n&:=-2\int_{t_n}^{t_{n+1}}\mathbb{E}\left[\mathbbm{1}_{\{Z_t\le0\}} \left( \kappa(\theta-\hat{v}_{\eta(t)})u^{\gamma}_v(t,\hat{x}_t,\hat{v}_t)+ \rho\sigma \hat{v}_{\eta(t)}u_{xv}^{\gamma}(t,\hat{x}_t,\hat{v}_t)\right)\right]dt,\\
		e^{(3,\gamma)}_n&:= \frac{1}{2}\int_{t_n}^{t_{n+1}}\mathbb{E}\left[(\hat{v}_t-\hat{v}_{\eta(t)})\mathcal{R}(u^{\gamma})(t,\hat{x}_t,\hat{v}_t)\right]dt ,\\
		e^{(4,\gamma)}_n&:= \frac{\gamma}{2} \int_{t_n}^{t_{n+1}}\mathbb{E}\left[\mathcal{R}(u^{\gamma})(t,\hat{x}_t,\hat{v}_t)\right]dt.
	\end{align*}
	
	\subsubsection{The first term}
	Recall that $\int_{t_n}^{t_{n+1}}u_v^{\gamma}(t,\hat{x}_t,\hat{v}_t)dL_t^0(Z)$ is a pathwise Riemann-Stieltjes integral  and $L^0(Z)$ is pathwise increasing. Therefore we have
	$$ \left| \int_{t_n}^{t_{n+1}}u_v^{\gamma}(t,\hat{x}_t,\hat{v}_t)dL_t^0(Z) \right| \leq  \sup\limits_{t\in[t_n,t_{n+1}]}|u_v^{\gamma}(t,\hat{x}_t,\hat{v}_t)| \left(L^0_{t_{n+1}}(Z)-L^0_{t_{n}}(Z)\right).$$
	With Lemma \ref{lembriani} it follows
	\begin{align*} \left| \int_{t_n}^{t_{n+1}}u_v^{\gamma}(t,\hat{x}_t,\hat{v}_t)dL_t^0(Z) \right| & \leq  C_f \sup\limits_{t\in[0,T]}
		\left(1+|\hat{x}_t|^a+|1+\hat{v}_t|^a\right) \left(L^0_{t_{n+1}}(Z)-L^0_{t_{n}}(Z)\right).
	\end{align*}
	The  Lemmata \ref{vbounded} and \ref{xbounded} yield the existence of  a constant $C_{f,p}>0$ such that
	\begin{align*}
		\left( \mathbb{E}\left|\sup\limits_{t\in[0,T]}\left(1+|\hat{x}_t|^a+|1+\hat{v}_t|^a\right)\right|^p\right)^{1/p} \leq C_{f,p},
	\end{align*}
	and H\"older's inequality then gives
	$$ |e^{(1,\gamma)}_n| \leq C_{f,\beta} 	\left( \mathbb{E}\left| L^0_{t_{n+1}}(Z)-L^0_{t_{n}}(Z)\right|^{1+\beta} \right)^{ \frac{1}{1+\beta}}$$
	for $\beta >0$.
	With Proposition \ref{proplocaltime}, we can therefore conclude that
		\begin{align}\label{est-1}
			|e_n^{(1,\gamma)}| \leq C_{f,\beta,\delta} \left(\Delta t\right)^{\frac{1}{(1+\beta)^2}} \left(\frac{\Delta t}{\varepsilon}\right)^{\nu\frac{1-\varepsilon}{(1+\delta)(1+\beta)^2}},
		\end{align} 
		uniformly in $\gamma \in (0,1]$.

	\subsubsection{The second term}
	Recall that
	\begin{align*}
		e^{(2,\gamma)}_n&=-2\int_{t_n}^{t_{n+1}}\mathbb{E}\left[\mathbbm{1}_{\{Z_t\le0\}} \left( \kappa(\theta-\hat{v}_{\eta(t)})u_v^{\gamma}(t,\hat{x}_t,\hat{v}_t)+ \rho\sigma \hat{v}_{\eta(t)}u_{xv}^{\gamma}(t,\hat{x}_t,\hat{v}_t)\right)\right]dt.
	\end{align*}
	An application of H\"older's inequality yields
	\begin{align*}
		|e^{(2,\gamma)}_n| \leq 2\int_{t_n}^{t_{n+1}}\left( P (Z_t\le0)  \right)^{\frac{1}{1+\beta}}
		\left( \mathbb{E} \left[ \left| \kappa(\theta-\hat{v}_{\eta(t)})u_v^{\gamma}(t,\hat{x}_t,\hat{v}_t)+ \rho\sigma \hat{v}_{\eta(t)}u_{xv}^{\gamma}(t,\hat{x}_t,\hat{v}_t)\right|^{\frac{1+\beta}{\beta}} \right] \right)^{\frac{\beta}{1+\beta}} dt.
	\end{align*}
	Lemma \ref{lembriani} and the Lemmata \ref{vbounded} and \ref{xbounded}  give that
	\begin{align*}
		2 \left( \mathbb{E} \left[ \left|  \kappa(\theta-\hat{v}_{\eta(t)})u_v^{{\gamma}}(t,\hat{x}_t,\hat{v}_t)+ \rho\sigma \hat{v}_{\eta(t)}u_{xv}^{{\gamma}}(t,\hat{x}_t,\hat{v}_t)\right|^{\frac{1+\beta}{\beta}} \right] \right)^{\frac{\beta}{1+\beta}} \leq C_{f,\beta}
	\end{align*} for $\beta>0$. Since 
	$$
	P(Z_t \leq 0) \leq  c\left(\frac{\Delta t}{\varepsilon}\right)^{\nu(1-\varepsilon)}, \qquad t \in (t_n,t_{n+1}),
	$$  by Proposition \ref{lemma3.6}, we end up with
	\begin{align} \label{est-2} |e_n^{(2,\gamma)}| \leq C_{f,\beta}  \Delta t\left(\frac{\Delta t}{\varepsilon}\right)^{\nu\frac{1-\varepsilon}{1+\beta}},
	\end{align}
	uniformly in $\gamma \in (0,1]$.
	
	\subsubsection{The third term}
	Now, we consider
	$$ e^{(3,\gamma)}_n= \frac{1}{2} \int_{t_n}^{t_{n+1}}\mathbb{E}\left[(\hat{v}_t-\hat{v}_{\eta(t)})\mathcal{R}(u^{\gamma})(t,\hat{x}_t,\hat{v}_t)\right]dt.$$
	Due to our assumptions the function $k^{\gamma}:=\mathcal{R}(u^{\gamma})$ belongs to  $C^1_{pol,T}$.
	Using the expression for $\hat{v}_t$ from Equation \eqref{itosym} we have
	\begin{align*}
		&\int_{t_n}^{t_{n+1}}\mathbb{E}\left[(\hat{v}_t-\hat{v}_{\eta(t)})k^{\gamma}(t,\hat{x}_t,\hat{v}_t)\right]dt\\
		& \quad = \int_{t_n}^{t_{n+1}}\mathbb{E}\left[\left(\int_{\eta(t)}^{t}\sign\left(Z_s\right)\kappa(\theta-\hat{v}_{\eta(s)})ds+\sigma\int_{\eta(t)}^{t}\sign\left(Z_s\right)\sqrt{\hat{v}_{\eta(s)}}dW_s\right)k^{\gamma}(t,\hat{x}_t,\hat{v}_t)\right]dt\\
		& \qquad +\int_{t_n}^{t_{n+1}}\mathbb{E}\left[\left(L_t^0\left(\hat{v}\right)-L_{\eta(t)}^0\left(\hat{v}\right)\right)k^{\gamma}(t,\hat{x}_t,\hat{v}_t)\right]dt.
	\end{align*}

	Looking at the first term, we have using H\"older's inequality, Equation \eqref{bound_rem_pde} and the Lemmata  \ref{vbounded}, \ref{xbounded} that
	\begin{align} \label{est3-1}
		\left| \mathbb{E}\left[\left(\int_{\eta(t)}^{t}\sign\left(Z_s\right) \kappa(\theta-\hat{v}_{\eta(s)})ds\right)k^{\gamma}(t,\hat{x}_t,\hat{v}_t)\right]\right|& \leq C_f \Delta t.
	\end{align}

	By an application of the law of total expectation, the H\"older and the Minkowski inequalities we have
	\begin{align*}
		&\left|\mathbb{E}\left[\int_{\eta(t)}^{t} \sign\left(Z_s\right) \sqrt{\hat{v}_{\eta(s)}}dW_sk^{\gamma}(t,\hat{x}_t,\hat{v}_t)\right]\right|\\ & \quad =\left|\mathbb{E}\left[\int_{\eta(t)}^{t} \sign\left(Z_s\right) \sqrt{\hat{v}_{\eta(s)}}dW_s \left( k^{\gamma}(t,\hat{x}_t,\hat{v}_t)-k^{\gamma}(t,\hat{x}_{\eta(t)},\hat{v}_{\eta(t)})\right)\right] \right|\\ 
		& \quad \le \mathbb{E}\left[\left|\int_{\eta(t)}^{t} \sign\left(Z_s\right)  \sqrt{\hat{v}_{\eta(s)}}dW_s\right|^2\right]^{\frac{1}{2}}  
		\\& \qquad \qquad \times \left[  \left( 
		\mathbb{E}\left| (k^{\gamma}(t,\hat{x}_t,\hat{v}_t)-k^{\gamma}(t,\hat{x}_{\eta(t)},\hat{v}_t) \right|^2  \right)^{1/2} 
		+   \left( 
		\mathbb{E}\left| (k^{\gamma}(t,\hat{x}_{\eta(t)},\hat{v}_t)-k^{\gamma}(t,\hat{x}_{\eta(t)},\hat{v}_{\eta(t)} \right|^2  \right)^{1/2} 
		\right].
	\end{align*}
	The mean value theorem now gives
	$$ k^{\gamma}(t,\hat{x}_t,\hat{v}_t)-k^{\gamma}(t,\hat{x}_{\eta(t)},\hat{v}_t)= \left( \int_0^1 k^{\gamma}_x(t, \lambda \hat{x}_t +(1-\lambda) \hat{x}_{\eta(t)},\hat{v}_t) d \lambda \right)   (\hat{x}_t-\hat{x}_{\eta(t)})
	$$
	and so
	\begin{align*}
		&   \left( \mathbb{E}\left|  k^{\gamma}(t,\hat{x}_t,\hat{v}_t)-k^{\gamma}(t,\hat{x}_{\eta(t)},\hat{v}_t)\right|^2 \right)^{1/2} \\ & \qquad \qquad \quad
		\leq \left( \int_0^1 \left( \mathbb{E}\left|   k^{\gamma}_x(t, \lambda \hat{x}_t +(1-\lambda) \hat{x}_{\eta(t)},\hat{v}_t) \right|^4 \right)^{1/4} d \lambda \right) \left(\mathbb{E}\left|  \hat{x}_t-\hat{x}_{\eta(t)} \right|^4 \right)^{1/4}.
	\end{align*}
	Equation \eqref{bound_rem_pde-2} and the Lemmata  \ref{vbounded}, \ref{xbounded} imply that
	$$  \sup_{t \in [0,T]}\int_0^1 \left( \mathbb{E}\left|   k^{\gamma}_x(t, \lambda \hat{x}_t +(1-\lambda) \hat{x}_{\eta(t)},\hat{v}_t) \right|^4 \right)^{1/4} d \lambda \leq C_f. $$
	Thus, we have by Lemma \ref{localxsymm}  that
	$$  \left( \mathbb{E}\left|  k^{\gamma}(t,\hat{x}_t,\hat{v}_t)-k^{\gamma}(t,\hat{x}_{\eta(t)},\hat{v}_t)\right|^2 \right)^{1/2} \leq C_f (\Delta t)^{1/2} .$$
	Similarly, we obtain
	$$  \left( \mathbb{E}\left|k^{\gamma}(t,\hat{x}_{\eta(t)},\hat{v}_t)-k^{\gamma}(t,\hat{x}_{\eta(t)},\hat{v}_{\eta(t)}) \right|^2  \right)^{1/2}  \leq C_{f} (\Delta t)^{1/2}.
	$$
	Since
	$$ \left[ \mathbb{E}\left|\int_{\eta(t)}^{t} \sign\left(Z_s\right)  \sqrt{\hat{v}_{\eta(s)}}dW_s\right|^2\right]^{\frac{1}{2}}  \leq C (\Delta t)^{1/2} $$
	by Lemma \ref{vbounded} and the It\=o-isometry, we end up with 
	\begin{align} \label{est3-2}\left| \mathbb{E}\left[\left(\sigma\int_{\eta(t)}^{t}\sign\left(Z_s\right) \sqrt{\hat{v}_{\eta(s)}}dW_s\right)k^{\gamma}(t,\hat{x}_t,\hat{v}_t)\right]\right| \leq C_f \Delta t.
	\end{align}
	
	With the H\"older inequality for some $\beta>0$, we have
	\begin{align*}
		& \left|\int_{t_n}^{t_{n+1}}\mathbb{E}\left[\left(L_t^0(Z)-L_{\eta(t)}^0(Z)\right)k^{\gamma}(t,\hat{x}_t,\hat{v}_t)\right]dt\right| \\ & \qquad \le \int_{t_n}^{t_{n+1}}
		\left( \mathbb{E}\left|L_t^0(Z)-L_{\eta(t)}^0(Z)\right|^{1+\beta} \right)^{\frac{1}{1+\beta}}
		\left(  \mathbb{E} |k^{\gamma}(t,\hat{x}_t,\hat{v}_t)|^{\frac{1+\beta}{\beta}} \right)^{\frac{\beta}{1+\beta}}dt.
	\end{align*}
	As before, we can show that there exists a constant
	$ C_{f,\beta}>0$,  such that
	$$ \sup_{t \in [0,T]} \left(  \mathbb{E} |k^{\gamma}(t,\hat{x}_t,\hat{v}_t)|^{\frac{1+\beta}{\beta}} \right)^{\frac{\beta}{1+\beta}} \leq C_{f,\beta}.$$
	Since
	$$ \mathbb{E}\left[\left|L_t^0(Z)-L_{\eta(t)}^0(Z)\right|^{1+\beta}\right]^{\frac{1}{1+\beta}} \leq  C_{\beta,\delta}
		\left(\Delta t\right)^{\frac{1}{(1+\beta)^2}} \left(\frac{\Delta t}{\varepsilon}\right)^{\nu\frac{1-\varepsilon}{(1+\delta)(1+\beta)^2}} , $$
	again by Proposition \ref{proplocaltime}, we obtain that
	\begin{align}\label{est3-3} \nonumber
		& \left|\int_{t_n}^{t_{n+1}}\mathbb{E}\left[\left(L_t^0(Z)-L_{\eta(t)}^0(Z)\right)k^{\gamma}(t,\hat{x}_t,\hat{v}_t)\right]dt\right| \\ & \qquad \le C_{f,\beta,\delta} \Delta t 
			\left(\Delta t\right)^{\frac{1}{(1+\beta)^2}} \left(\frac{\Delta t}{\varepsilon}\right)^{\nu\frac{1-\varepsilon}{(1+\delta)(1+\beta)^2}}
		.
	\end{align}

	Summarizing \eqref{est3-1}, \eqref{est3-2} and \eqref{est3-3} we have shown that
			\begin{align}
				|e_n^{(3,\gamma)}| \leq C_{f,\beta, \delta}  \Delta t \left(
				\Delta t   + 
				\left(\Delta t\right)^{\frac{1}{(1+\beta)^2}} \left(\frac{\Delta t}{\varepsilon}\right)^{\nu\frac{1-\varepsilon}{(1+\delta)(1+\beta)^2}}
				\right), 
	\end{align} 
	uniformly in $\gamma \in (0,1]$.
	
	\subsubsection{The fourth term}
	Finally, consider
	$$ e^{(4,\gamma)}_n= \frac{\gamma}{2} \int_{t_n}^{t_{n+1}}\mathbb{E}\left[\mathcal{R}(u^{\gamma})(t,\hat{x}_t,\hat{v}_t)\right]dt.$$
	Since
	$$ \sup_{t \in [0,T]} \mathbb{E}\left[\mathcal{R}(u^{\gamma})(t,\hat{x}_t,\hat{v}_t)\right] \leq C_{f}$$ 
	due to Equation \eqref{bound_rem_pde} and the Lemmata  \ref{vbounded}, \ref{xbounded}, 
	we have that
	\begin{align}
		\frac{1}{\gamma} |e^{(4,\gamma)}_n| \leq C_f \Delta t, \end{align}
	uniformly in $\gamma \in (0,1]$.
	
	\subsubsection{The conclusion}
			Recall that $\Delta t =T /N$. 
			Adding the estimates for term one up to term four we have derived that
			\begin{align*}
				|e_n^{\gamma} | & \leq C_{f,\beta, \delta}  	\left(\Delta t\right)^{\frac{1}{(1+\beta)^2}} \left(\frac{\Delta t}{\varepsilon}\right)^{\nu\frac{1-\varepsilon}{(1+\delta)(1+\beta)^2}}
				\\ & \qquad  +C_{f,\beta}  \Delta t
				\left( \frac{\Delta t}{\varepsilon} \right)^{\nu \frac{1-\varepsilon}{1+\beta}} 
				\\& \qquad +  C_{f,\beta, \delta}  \Delta t \left(
				\Delta t   + 
				\left(\Delta t\right)^{\frac{1}{(1+\beta)^2}} \left(\frac{\Delta t}{\varepsilon}\right)^{\nu\frac{1-\varepsilon}{(1+\delta)(1+\beta)^2}}
				\right)
				\\ &  \qquad + C_f \gamma \Delta t. 
			\end{align*}
			For any given $\epsilon \in( 0, 1/2)$ we now can find $\varepsilon \in (0,1/2]$, $\beta >0$ and $\delta >0$ such that
			$$ \frac{1}{(1+\beta)^2} + \nu\frac{1-\varepsilon}{(1+\delta)(1+\beta)^2} \geq 1+ \nu(1-\epsilon)$$
			and
			$$ 1+ \nu \frac{1-\varepsilon}{1+\beta}  \geq 1+ \nu(1-\epsilon).$$
			Consequently, we obtain
			\begin{align*}
				|e_n^{\gamma} | & \leq C_{f,\epsilon} \Delta t  	\left(  \Delta t + (\Delta t)^{ \nu(1-\epsilon)} \right)
				+ C_f \gamma \Delta t
			\end{align*}
			and
			\begin{align*}
				\sum_{n=0}^{N-1} |e_n^{\gamma}| \leq C_f \gamma +  C_{f,\epsilon} \left(  \Delta t + (\Delta t)^{ \nu(1-\epsilon)} \right). \end{align*}
			Since $$|\EE \left[f(x_N,v_N)- f(X_T,V_T)\right]| \leq \limsup_{\gamma \searrow 0}\sum_{n=0}^{N-1} |e_n^{\gamma}|$$
			we have that
			\begin{align*}
				\label{final}
				|\EE \left[f(x_N,v_N)- f(X_T,V_T)\right]| \leq 	C_{f, \epsilon}\left(  \Delta t + (\Delta t)^{ \nu(1-\epsilon)} 
				\right),
			\end{align*}
			which concludes the proof. 
	
	\bigskip

	\smallskip
	\subsection{The absorbed Euler scheme}
	All tools and auxiliary results have been derived simultaneously for the symmetrized and the absorbed Euler scheme. The only difference is that we have
	\begin{equation*}
		\begin{aligned}
			\hat{v}^{abs}_t=\text{ }v_{n(t)}^{abs} &+\int_{\eta(t)}^{t}\mathbbm{1}_{\{Z_s^{abs}>0\}}\kappa\left(\theta-v_{n(s)}^{abs}\right)ds+\sigma\int_{\eta(t)}^{t}\mathbbm{1}_{\{Z_s^{abs}>0\}}\sqrt{v_{n(s)}^{abs}}dW_s \\ & +\frac{1}{2}\left( L_t^0(Z^{abs})-L_{\eta(t)}^0(Z^{abs})\right)
		\end{aligned}
	\end{equation*}
	instead of 
	\begin{equation*}
		\begin{aligned}
			\hat{v}^{sym}_t=v_{n(t)}^{sym} & +\int_{\eta(t)}^{t} \sign\left(Z_s^{sym}\right)\kappa\left(\theta-v_{n(s)}^{sym}\right)ds+\sigma\int_{\eta(t)}^{t}\sign\left(Z_s^{sym}\right)\sqrt{v_{n(s)}^{sym}}dW_s \\ & +\left( L_t^0(Z^{sym})-L_{\eta(t)}^0(Z^{sym})\right).
		\end{aligned}
	\end{equation*}
	Besides this difference,  we can carry out the proof completely analogous.

	\bigskip
	\bigskip
	
	{\bf Acknowledgments.} \,\,	
	{\it The authors are very thankful to the referees for their insightful comments and remarks, which greatly helped to improve this manuscript.}
	{\it  Annalena Mickel has been supported by the DFG 
		Research Training Group 1953 "Statistical Modeling of Complex Systems".}

\end{document}